\documentclass[reqno, oneside]{amsart}

\usepackage{Style} 

\title[A geometric and generating function approach to plethysm]{A geometric and generating function\\ approach to plethysm}
\author[Á. Gutiérrez, R. Orellana, F. Saliola, A. Schilling, M. Zabrocki]{Álvaro Gutiérrez, Rosa Orellana, Franco Saliola,\\ Anne Schilling, Mike Zabrocki}
\date{\today}
\begin{document}

\begin{abstract}
Plethysm coefficients $\mathsf{a}_{\mu[\nu]}^\lambda$ are the structure coefficients of the plethysm
of Schur functions $s_\mu[s_\nu] = \sum_{\lambda} \mathsf{a}_{\mu[\nu]}^\lambda s_\lambda$.  We study
a bivariate generating function of plethysm coefficients when $\lambda$ has bounded length. We show
that this generating function is rational.   A key step is MacMahon's combinatory analysis.
When the bound on the length is $2$ we give an explicit geometric algorithm to compute it
using $q$-Ehrhart theory. We give evidence that the
generating function is the quantum Ehrhart series of a union of half-open polytopes and show that
it satisfies a reciprocity theorem reminiscent of Ehrhart reciprocity. Furthermore, we give a set
of linear recursions that completely describe the $\mathrm{SL}_2$-plethysm coefficients.
\end{abstract}

\maketitle

\vspace{-1em}

\tableofcontents

\section{Introduction}
\label{sec: introduction}

Given two polynomial representations of the general linear group
$\rho:\operatorname{GL}_n\to\operatorname{GL}_m$ and
$\tau:\operatorname{GL}_m\to\operatorname{GL}_k$, we can compose them to obtain a
new representation of the general linear group $\tau\circ \rho:\operatorname{GL}_n\to\operatorname{GL}_k$.
This operation was first studied by Littlewood~\cite{Littlewood.1936} who
described it in terms of symmetric functions:
if $f$ and $g$ are the symmetric functions corresponding to the characters
of $\rho$ and $\tau$, respectively, then their \emph{plethysm} is
the symmetric function $g[f]$ corresponding to the character of $\tau \circ \rho$.
The \emph{plethysm coefficients} $\aaa_{\mu[\nu]}^\lambda$ are the structure constants of the plethysm of Schur functions,
\[
s_\mu[s_\nu] = \sum_{\lambda} \aaa_{\mu[\nu]}^\lambda \, s_\lambda.
\]
A major open problem as identified by Stanley~\cite{Stanley.2000} is to find a
combinatorial interpretation of $\aaa_{\mu[\nu]}^\lambda$. A solution to this
problem has important applications in physics~\cite{Wybourne.1970}, 
invariant theory~\cite{Howe.1987}, and complexity theory~\cite{MulmuleySohoni.2003}.
Various special cases have been solved, going back to the work of Littlewood in the 40's (see
for instance~\cite{Littlewood.1950,Carre.Leclerc.1995,Carini.Remmel.1998,LR.2004,Loehr.Remmel.2011,PW.2019,
OSSZ.2022,GR.2023,COSSZ.2022,BP.2024,OSSZ.2024,GutCrystals,PakPanovaSwanson.2025}
and references therein), but a general formula or interpretation for $\aaa_{\mu[\nu]}^\lambda$ is still elusive.
These generating functions were shown to be rational functions by  Mulmuley~\cite{GCT6},
who conjectured that they are Ehrhart series of (unions of) closed 
polytopes. Kahle and Michałek~\cite{KahleMichalek18} disproved this conjecture
by arguing that it would violate Ehrhart reciprocity for closed polytopes.
We note in Section~\ref{sec:KM} that a weakening of Mulmuley's conjecture is
still open.
We remark that the use of generating functions to study representation-theoretic
coefficients has appeared in other works; see for example~\cite{PS, DG,
KingWelsh, BBDGK, KahleMichalek18}.

We revisit this approach by introducing two
generating functions of plethysm coefficients: for a partition $\mu$,
\begin{equation}
\label{equation.Amu GLn}
\mathbb{A}_\mu(x_1,\ldots,x_n;y_1,\ldots,y_m)
= \sum_{\ell(\lambda)\leqslant n} \sum_{\ell(\nu)\leqslant m} \aaa_{\mu[\nu]}^\lambda \, x^\lambda \, y^\nu,
\end{equation}
and for partitions $\mu$ such that its length $\ell(\mu) \leqslant n$,
\begin{equation}
\label{equation.Bmu GLn}
\mathbb{B}_\mu(x_1,\ldots,x_n;y_1,\ldots,y_m)
= \sum_{\ell(\lambda)\leqslant n} \sum_{\ell(\nu)\leqslant m} \aaa_{\nu[\mu]}^\lambda \, x^\lambda \, y^\nu,
\end{equation}
where $x^\lambda = x_1^{\lambda_1}x_2^{\lambda_2}\cdots x_n^{\lambda_{n}}$, and $y^\nu$ is similarly defined. 
We show that both of these generating functions are rational, broadening Mulmuley's result.

The initial part of this paper is devoted to a special case of \eqref{equation.Amu GLn}.
When $n = 2$ and $m = 1$, the coefficients $\aaa_{\mu[\nu]}^\lambda$
appearing in \eqref{equation.Amu GLn} are indexed by $\lambda$ with
at most two parts and $\nu$ with at most one part.
Writing $\lambda = (\lambda_1, \lambda_2)$ and $\nu = (h)$,
with $\lambda_1, \lambda_2, h \in \mathbb{N}$, we have
\begin{equation}
\label{equation.Amu GLn-special-case}
\mathbb{A}_\mu(x_1,x_2;y_1)
= \sum_{\lambda_1 \geqslant \lambda_2 \geqslant 0} 
\sum_{h \geqslant 0} \aaa_{\mu[(h)]}^{(\lambda_1, \lambda_2)} \, x_1^{\lambda_1} \, x_2^{\lambda_2} \, y_1^{h}.
\end{equation}
These coefficients admit an alternative description.
As explained in Section~\ref{SL2-plethysm-coefficients},
$s_\mu[s_{(h)}](q,q^{-1})$ can be expressed uniquely as
\begin{equation}
    s_\mu[s_{(h)}](q,q^{-1}) = \sum_{k \geqslant 0}  a_{\mu[h]}^{[k]} \, [k]_q,
\end{equation}
where the $a_{\mu[h]}^{[k]}$ are coefficients in $\mathbb{N}$, and
the $[k]_q$ are the \emph{quantum} or \emph{symmetric} \emph{$q$-integers} defined as
\begin{equation}
\label{equation.q integer}
	[k]_q := \frac{q^k-q^{-k}}{q-q^{-1}} \quad \text{for $k \geqslant 0$.}
\end{equation}
Then it turns out that
\[
	\aaa_{\mu[(h)]}^{(\lambda_1,\lambda_2)} =a_{\mu[h]}^{[\lambda_1-\lambda_2+1]}.
\]
Since these coefficients depend only on the difference $\lambda_1 - \lambda_2$,
we study the generating function
\begin{equation}
\label{equation.Amu}
	A_\mu(z, q) := q \mathbb{A}_{\mu}(q, q^{-1}; z)
    = \sum_{k\geqslant 1} \sum_{h\geqslant 0} a_{\mu[h]}^{[k]} \, q^k \, z^h.
\end{equation}
We refer to the coefficients $a_{\mu[h]}^{[k]}$ as \emph{$\SL_2$-plethysm
coefficients} due to their connection with characters of $\SL_2$.

In this paper we show that $A_\mu(z,q)$ is a rational generating
function. We do this by combining several results from Ehrhart theory and
MacMahon's combinatory analysis.
We aim to give a geometric account of each step of our derivation as properties
of $A_\mu(z,q)$ suggest it is a geometric object (cf.~Conjecture~\ref{conj:Ehr}
and Section~\ref{section.reciprocity}).
Moreover, our results on $A_\mu(z, q)$ demonstrate that the
$\SL_2$-plethysm coefficients possess hidden structures.
Our main results are as follows.
\begin{enumerate}[leftmargin=*, topsep=2ex, itemsep=2ex]
\item
We show (see Theorem~\ref{th:semi-final}) that there exists a polynomial $p_\mu(z,q)\in\ZZ[q^\pm,z]$
such that
\[
A_\mu(z,q) = \frac{p_\mu(z,q)}{d_{|\mu|}(z,q)},
\qquad\text{where}\quad
    d_w(z,q) :=
    \begin{cases}
        (1-z)\prod_{i=1}^w (1-z^i) \prod_{i=1}^{\frac{w}{2}} (1-q^{2i} z), & \text{for $w$ even,}\\[1ex]
        \prod_{i=1}^w (1-z^{2i}) \prod_{i=1}^{\frac{w+1}{2}} (1-q^{2i-1} z), & \text{for $w$ odd.}
    \end{cases}
\]
We conjecture a denominator with fewer terms (see Conjecture~\ref{conj.good denominator}).
\item
For $|\mu|\leqslant 5$, we express $A_\mu(z,q)$ as a sum of
\emph{positive} rational generating functions
(see Section~\ref{ss:GeneratingFunctions} and Appendix~\ref{appendix.A}).
This makes $A_\mu(z,q)$ the 
weighted generating function of the lattice points of a disjoint union of polytopes,
which allows us to obtain a combinatorial interpretation of the associated $\SL_2$-plethysm coefficients.
We believe these positive expansions should exist more generally; more precisely, we conjecture the following,
which reflects the philosophy of Mulmuley's conjecture, but neither implies it nor is implied by it.
\begin{conjecture}\label{conj:Ehr}
    The generating function $A_\mu(z,q)$ is a weighted generating function of lattice points of
    a disjoint union of convex rational half-open cones.
\end{conjecture}
\noindent
\item
We show that $A_\mu(z,q)$ satisfies a reciprocity theorem (see Theorem~\ref{thm:reciprocity}) that is reminiscent of ($q$-)Ehrhart reciprocity 
for half-open polytopes:
\begin{equation}
\label{equation.reciprocity}
	A_\mu(z^{-1},q^{-1}) = (-1)^{w}z^{2} A_{\mu'}(z,q),
\end{equation}
where $\mu'$ is the conjugate of $\mu$. This gives evidence for Conjecture~\ref{conj:Ehr}.

Note that if $\mu$ is self-conjugate then~\eqref{equation.reciprocity} makes $A_\mu$ ``self-reciprocal'', in the sense that
\[
\frac{A_\mu(z^{-1},q^{-1})}{A_{\mu}(z,q)}
\]
is just a power of $z$, up to a sign. In Corollary~\ref{cor:hooks} we establish this phenomenon for a larger family of partitions which includes also 
hooks, rectangular partitions, and when $w$ is even, partitions of $w$ with empty $(w/2)$-core.

\item
We give a set of recurrences that completely describe $\SL_2$-plethysm coefficients (see Section~\ref{ss:recursion}).
These results generalize our previous enumerative results obtained in \cite{OSSZ.2024} and \cite{GutCrystals}. In these papers, 
we obtained recursions satisfied by some $\SL_2$-plethysm coefficients and combinatorial interpretations for the case $\mu = (3)$ and $\mu=(4)$
by finding symmetric chain decompositions of the Young lattice.

Note that for $\mu=(w)$, the coefficients $a_{\mu[h]}^{[k]}$ are also certain Kronecker coefficients~\cite{PakPanova.2013,Panova.2023}.
A combinatorial interpretation for the $\SL_2$-plethysm coefficients is given in~\cite{PakPanovaSwanson.2025} in
terms of computational trees, by a detailed analysis of a generalization of the KOH formula \cite{OHara,Zeilberger.1989,GOS}, which is
different from the approach and results in this paper.

\item We outline a new computational approach towards the Foulkes conjecture. This conjecture states that if $a>b$ then 
$\langle s_a[s_b] - s_b[s_a], s_\lambda\rangle\geqslant 0$ for all partitions $\lambda$ of $ab$. It is known to hold for $b\leqslant 5$ \cite{CIM}. 
For each fixed $b$, we reduce the problem to a finite computation; a proof of the conjecture for $b=3$ is included in Appendix~\ref{appendix.B}. 
The main obstacle to push this result further is the speed of existing algorithms to express a rational function as a sum of positive rational functions.
\end{enumerate}

The paper is organized as follows. In Section~\ref{s:GeometricPicture}, we provide the geometric interpretation for
$A_\mu(z,q)$ culminating in Theorem~\ref{theorem.A QEhr} relating it to a $q$-Ehrhart series via the MacMahon operator.
In Section~\ref{section.combinatorial}, we deduce combinatorial facts from these results and give linear recursions for the $\SL_2$-plethysm coefficients.
In Section~\ref{section.Amu}, we provide explicit expressions for the denominator of $A_\mu(z,q)$.
In Section~\ref{section.reciprocity} we give a proof of the reciprocity in~\eqref{equation.reciprocity}.
We conclude in Section~\ref{section.GLn} with the analysis of the generating functions~\eqref{equation.Amu GLn} 
and~\eqref{equation.Bmu GLn} for $\GL_n$-plethysm coefficients.

\subsection*{Acknowledgements}
The authors would like to thank Matthias Beck, Giulia Codenotti, Jesús De Loera, Darij Grinberg, Yifan Guo, Christian Ikenmeyer, Carly Klivans, 
Greta Panova, Allan P\'erez, Brendon Rhoades, Josh Swanson, Monica Vazirani, Ole Warnaar, Geordie William\-son, and Guoce Xin
for many enlightening discussions. In particular, we thank Jesús De Loera for confirming that~\cite[Thm. 2.9]{DLVVW.2024} holds for
half-open polytopes, and Guoce Xin for providing the proofs of Lemma~\ref{lemma.partial fraction} and Theorem~\ref{thm:reciprocity}.

This material is based upon work supported by the National Science Foundation under Grant No. DMS-1928930, while the last four authors
were in residence at the Simons Laufer Mathematical Sciences Institute in Berkeley, California, during the 2025 Summer Research in 
Mathematics program.
This material is based upon work supported by the National Science Foundation under Grant No. DMS-1929284 while the authors were 
in residence at the Institute for Computational and Experimental Research in Mathematics in Providence, RI, during the
``Categorification and Computation in Algebraic Combinatorics'' semester program in Fall 2025.
Any opinions, findings, and conclusions or recommendations expressed in this material are those of the authors and do not necessarily reflect the views 
of the National Science Foundation.

ÁG was funded by a University of Bristol Research Training Support Grant.
AS was partially supported by  NSF grant DMS--2053350 and Simons Foundation grant MPS-TSM-00007191.
RO was partially NSF grants  DMS--2153998 and DMS--2452044.
FS acknowledges the support of the Natural Sciences and Engineering Research Council of Canada (NSERC) funding reference number RGPIN-2023-04476.
MZ was supported by NSERC/CNSRG.

\section{Geometric picture}
\label{s:GeometricPicture}

In this section, we develop a geometric approach to Schur polynomials and the generating function $A_\mu(z,q)$ of 
$\SL_2$-plethysm coefficients defined in~\eqref{equation.Amu} by combining modern interpretations of classical results from different areas: 
Ehrhart theory, (quasi)symmetric functions, and MacMahon's combinatory analysis. We rely on concepts developed 
in~\cite{BJR,GPS}, \cite{Chapoton,BK-Chapoton}, and~\cite{BZ-Omega}.

While individual results in this section are well known, their connection to plethysm and the geometric interpretation of the Schur polynomials 
is not found in the literature.  Our new result in this section is Theorem~\ref{theorem:integer_point_enumerator}.

\subsection{Preliminaries}
\label{section.notation}

We introduce some notation. See~\cite{EC2} for more details.

A \emph{partition} $\lambda=(\lambda_1,\ldots,\lambda_\ell)$ of a nonnegative integer $w$, denoted $\lambda \vdash w$, satisfies
$\lambda_1 \geqslant \lambda_2 \geqslant \cdots \geqslant \lambda_\ell >0$ for $\lambda_i$ integers and $\lambda_1+\cdots+\lambda_\ell=w$. 
We call $\ell(\lambda):=\ell$ the length of $\lambda$. A \emph{cell} $c=(i,j)$ in $\lambda$ satisfies $1\leqslant i \leqslant \ell$ and $1\leqslant j \leqslant
\lambda_i$.
For any cell $c \in \lambda$, denote the \emph{hook length} of $c$ in the partition $\lambda$ by $\mathsf{hl}_\lambda(c)$, which
is the number of cells in $\lambda$ strictly to the right of $c$ (meaning with larger column index) plus the number of cells in $\mu$ strictly below $c$ plus 1 
(meaning with bigger row index than $c$). The \emph{Durfee square} of the partition $\lambda$ is the largest $1\leqslant s\leqslant \ell$ such that 
$\lambda_s\geqslant s$.

For $\lambda$ a partition, denote by $\SSYT_{h+1}(\lambda)$ the set of \emph{semistandard Young tableaux}
of shape $\lambda$ over the alphabet $\{0,1,\ldots,h\}$. For each $\lambda$, the \emph{Schur polynomial} is defined as
\[
	s_\lambda(x_0,\ldots,x_h) = \sum_{T\in \SSYT_{h+1}(\lambda)} x^{\wt(T)},
\]
where $\wt(T)=(\alpha_0,\ldots,\alpha_h)$ is the \emph{weight} of $T$ with $\alpha_i$ equal to the number of letters $i$ in $T$ and
$x^{\wt(T)} = x_0^{\alpha_0} \cdots x_h^{\alpha_h}$. We denote by $\shape(T)=\lambda$ the shape of $T$.

The \emph{power sum} is defined as
\[
	p_k(x_0,\ldots,x_h) = \sum_{i=0}^k x_i^k
\]
and $p_\lambda = p_{\lambda_1} \cdots p_{\lambda_\ell}$ for a partition $\lambda=(\lambda_1,\ldots,\lambda_\ell)$.

A \emph{composition} $\beta=(\beta_1,\ldots,\beta_\ell)$ of a nonnegative integer $w$, denoted $\beta \models w$,
satisfies $\beta_1+\cdots + \beta_\ell=w$ and $\beta_i$ are positive integers.
The \emph{monomial quasisymmetric polynomial} indexed by a composition $\beta=(\beta_1,\ldots,\beta_\ell)$ with $\ell\leqslant h+1$
is defined as
\[
	M_\beta(x_0,\ldots,x_h) = \sum_{0\leqslant i_1<\cdots<i_\ell\leqslant h} x_{i_1}^{\beta_1}\cdots x_{i_\ell}^{\beta_\ell}.
\]
Gessel~\cite{Gessel.1984} introduced what are now known as \emph{Gessel's fundamental quasisymmetric function} as
\[
	F_{\alpha}(x_0,\ldots,x_h) = \sum_{\beta \preccurlyeq \alpha} M_\beta(x_0,\ldots,x_h),
\]
where $\alpha$ is also a composition and $\beta \preccurlyeq \alpha$ indicates that $\beta$ is a refinement of $\alpha$.

Let $\SYT(\lambda)$ be the set of \emph{standard Young tableaux} of shape $\lambda$. Then Gessel~\cite{Gessel.1984} proved that
\[
	s_\lambda(x_0,\ldots,x_h) = \sum_{T\in \SYT(\lambda)} F_{\Des(T)}(x_0,\ldots,x_h).
\]
Here $\Des(T)$ is the descent set of $T$, where $i$ is a \emph{descent} if the letter $i+1$ appears in larger (lower) row than
$i$ in English convention for partitions. If $\Des(T)=\{d_1,\ldots,d_k\}$ is the descent set of $T$, then we can associate to it the descent
composition $\alpha=(\alpha_1,\ldots,\alpha_k)$, where $\alpha_i = d_i-d_{i-1}$ and by convention $d_0=0$. By abuse of notation
$F_{\Des(T)} = F_\alpha$.

Similar to the descent set, we define the \emph{ascent set} $\Asc(T)$ as the letters $i$ in $T\in \SYT(\lambda)$ such $i+1$ is strictly to the 
right of $i$. Ascents and descents can also be defined for permutations $\pi\in S_w$, where $S_w$ is the symmetric group on $w$ elements.
A letter $i\in \pi$ is a descent (resp. ascent) if $i+1$ lies to the left (resp. right) of $i$. We denote by $\des(T)$ (resp. $\asc(T)$) the cardinality
of $\Des(T)$ (resp. $\Asc(T)$) and similarly for permutations.

The \emph{major index} of a tableau $T\in \SYT(\lambda)$ is
\[
	\maj(T) = \sum_{i\in \Des(T)} i
\]
and similarly $\maj(\pi) = \sum_{i\in \Des(\pi)} i$ for $\pi\in S_w$.
If we denote by $\lambda'$ (resp. $T'$) the conjugate of the partition $\lambda$ (resp. tableau $T\in\SYT(\lambda)$, then note that
\[
	\des(T') = |T|-1-\des(T) \quad \text{and} \quad \maj(T') = \binom{|T|}{2} - \maj(T).
\]

\subsection{\texorpdfstring{$\SL_2$}{SL2}-plethysm coefficients}
\label{SL2-plethysm-coefficients}

The character ring of $\SL_n$ is $\mathbb{Z}_{\geqslant 0}[x_1,x_2,\ldots,x_n]^{S_n} /(x_1 \cdots x_n -1)$, where 
$\mathbb{Z}_{\geqslant 0}[x_1,\ldots, x_n]^{S_n}$ denotes the symmetric polynomials in $n$ variables
with nonnegative integer coefficients. In particular, the character ring of $\SL_2$ is identified with the 
$\mathbb{Z}_{\geqslant 0}$-span of the $q$-integers as defined in~\eqref{equation.q integer}.
Define the \emph{$q$-binomial} as
\[
\Qbinom{n}{k}
=
\frac{[n]_q[n-1]_q\cdots[n-k+1]_q}{[k]_q[k-1]_q\cdots[1]_q}
\]
for $0 \leqslant k \leqslant n$, and $\Qbinom{n}{k} = 0$ otherwise. 

As shown in~\cite[\S I.8.~Example 4, p.~137]{Macdonald},
for $w, h \in \mathbb{N}$,
\[
	s_w[s_h](q^{-1},q)= \Qbinom{w+h}{w} = \sum_{k\geqslant 1} a_{w[h]}^{[k]} \, [k]_q
\]
for some nonnegative integer coefficients $a_{w[h]}^{[k]}$.
More generally, we write for a partition $\mu$
\[
	s_\mu[s_h](q^{-1},q) = \sum_{k\geqslant 1} a_{\mu[h]}^{[k]} \, [k]_q,
\]
where $a_{\mu[h]}^{[k]}$ is the coefficient of $[k]_q$ in $s_\mu[s_h](q^{-1},q)$,
which we call a \emph{$\SL_2$-plethysm coefficient}.

Up to an overall power of $q$ and replacing
$q$ by $q^2$, the plethysm $s_\mu[s_h](q^{-1},q)$ is given by the principle specialization $s_\mu(1,q,q^2,\ldots,q^h)$ which 
in turn is related to the generalized Gaussian polynomial~\cite[p.~137]{Macdonald}.

\subsection{The integer-point enumerator and Schur polynomials}

Let $V = \C^{h+1}$ denote the natural representation of $\SL_{h+1}(\C)$.
The character of $V$ is given by $x_0 + \cdots + x_h$.
The character of the tensor power $V^{\otimes w}$ is given by $(x_0+\cdots+x_h)^w$.
A corollary of the RSK algorithm \cite[Cor.~7.12.5]{EC2} and a consequence of Stanley's theory of $P$-partitions \cite[Thm.~7.19.7]{EC2} give the equalities
\begin{equation}
\label{eq:character of cube}
(x_0+\cdots+x_h)^w = \sum_{\lambda\vdash w} f^\lambda s_\lambda(x_0, \ldots, x_h) = \sum_{\pi \in S_w} F_{\Des(\pi)}(x_0, \ldots, x_h),
\end{equation} 
where $f^\lambda$ is the number of standard Young tableaux of shape $\lambda$.
Since we work over $h+1$ variables, we automatically have that the partitions $\lambda$ have at most $h+1$ parts.
Our first goal is to paint a geometric picture for~\eqref{eq:character of cube}, which loosely follows \cite{BJR, GPS}.\medskip

Recall that the RSK algorithm gives a bijection between words $v_1 v_2 \cdots
v_w$ of length $w$ in the alphabet $\{0,1,\ldots,h\}$ and pairs of tableaux
in $\bigcup_{\lambda\vdash w} \SSYT_{h+1}(\lambda)\times\SYT(\lambda)$; see for example~\cite{SaganBook,EC2}.
Now consider the $w$-dimensional unit hypercube $\square = [0,1]^w$.
By identifying the word $v_1 v_2 \cdots v_w$ with the vector $(v_1, v_2, \ldots, v_w)$,
the RSK algorithm can be interpreted as a bijection between 
the integral points of the dilated hypercube $h\square$
and the set $\bigcup_{\lambda\vdash w} \SSYT_{h+1}(\lambda)\times\SYT(\lambda)$.
We refine this bijection as follows.

For $1 \leqslant i < j \leqslant w$, the hyperplane
$\{ v_i = v_j\} := \{ (v_1, \ldots, v_w) : v_i = v_j\}$
divides the cube $\square$ into two regions, 
\[
\square\cap\{v_i \leqslant v_j\}, ~~\text{and}~~
\square\cap\{v_i > v_j\}.
\]
The \emph{braid hyperplane arrangement} $\bigcup_{1\leqslant i < j \leqslant w} \{v_i = v_j\}$ then divides the hypercube into $w!$ regions 
(\emph{chambers}), each chamber naturally indexed by a total order on $\{v_1, \ldots, v_w\}$. The closures of these chambers are usually called
\emph{Weyl chambers} in the literature; however, we work with sets defined by weak and strict inequalities. That is to say, we want to work with 
the half-open polytopes defined by the hyperplanes.
A permutation $\pi\in S_w$ corresponds naturally to an order $(v_{\pi(1)}, \ldots, v_{\pi(w)})$. From now on we denote by $\chamber(\pi)$ the 
half-open chamber corresponding to the permutation~$\pi$ and often drop the qualifier ``half-open'' when referring to these regions.
Explicitly, $\chamber(\pi)$ is the half-open polytope
\begin{align*}
    \chamber(\pi) &= \{
    v\in \square \mid
    v_{\pi(i)} \leqslant v_{\pi(i+1)} \text{~if~} \pi(i)<\pi(i+1),~~
    v_{\pi(i)} < v_{\pi(i+1)} \text{~otherwise}
    \}\\
    &= \{
    v\in \square \mid
    v_{\pi(i)} \leqslant v_{\pi(i+1)} \text{~if~} i\in\Asc(\pi),~~
    v_{\pi(i)} < v_{\pi(i+1)} \text{~if~} i\in\Des(\pi)
    \}.
\end{align*}
We denote by $h\cdot\chamber(\pi)$ the dilated chamber with point in $v\in h\square$ and otherwise the same conditions.

\begin{example}\label{ex:chamber}
    Let $w=6$. Then
    \[
    \chamber(231456) = \{v\in \square \mid v_2 \leqslant v_3 < v_1 \leqslant v_4 \leqslant v_5 \leqslant v_6\}.
    \]
\end{example}

\begin{example}
    Let $w=3$. The integer points of the $3!$ chambers of $9\square$ are depicted in Figure \ref{fig:chambers}.
    To its right, a generic slice of the cube separating the points in chambers.
    The set $\{v_1\leqslant v_2 \leqslant v_3\}$ is $\chamber(123)$, the set $\{v_2 < v_1 \leqslant v_3\}$ is $\chamber(213)$, and so on.
\end{example}

    \begin{figure}[h]
        \centering
        \includegraphics[width=0.3\linewidth]{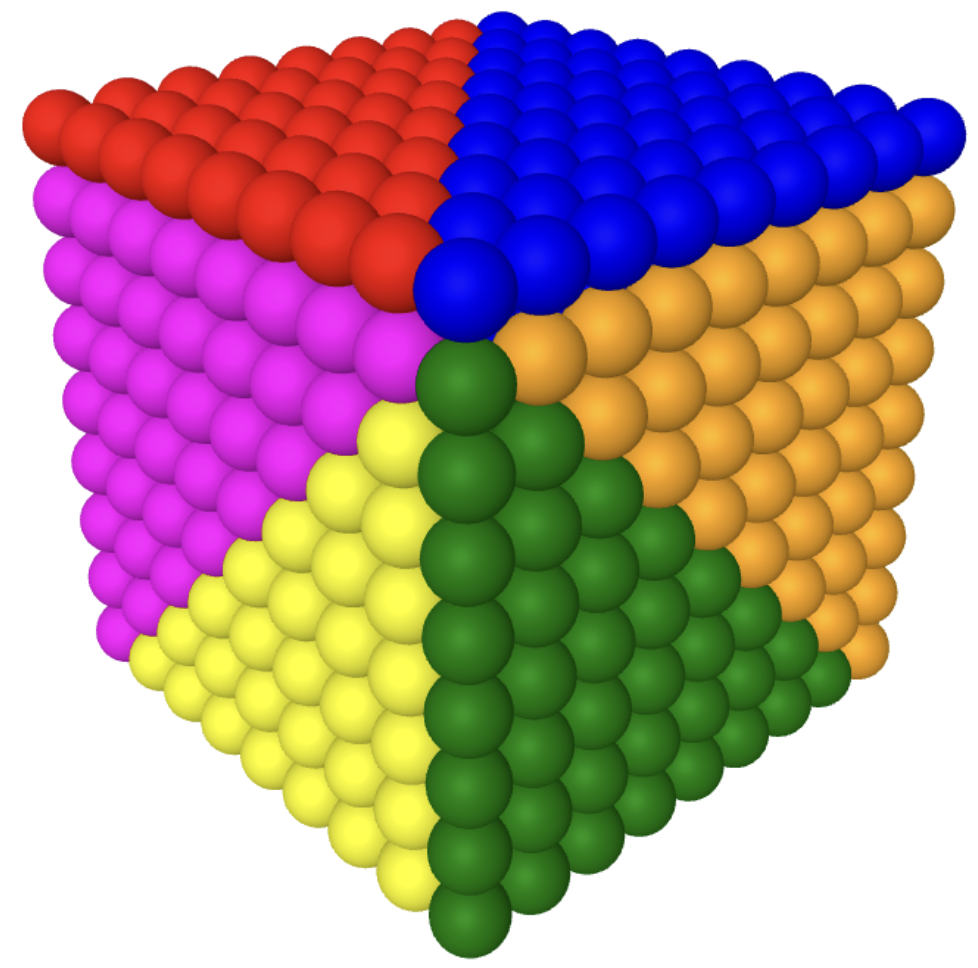}
        \begin{tikzpicture}[x=2.5em,y=2.5em,rotate=30]
            \begin{scope}[line width=6, rounded corners=1pt]
            \filldraw[blue!50] (0,0) -- ++(0:2) -- ++(120:2) -- cycle;
            \filldraw[red!60] (-.2,.34) -- ++(60:1.6) -- ++(180:1.6) -- cycle;
            \filldraw[orange!70] (.2,-.34) -- ++(-60:1.6) -- ++(60:1.6) -- cycle;
            \filldraw[magenta!70] (-.4,0) -- ++(180:1.6) -- ++(60:1.6) -- cycle;
            \filldraw[green!70!black!70] (-.2,-.34) -- ++(-60:1.6) -- ++(180:1.6) -- cycle;
            \filldraw[yellow!90!orange!70] (-.6,-.34) -- ++(180:1.2) -- ++(-60:1.2) -- cycle;
            \end{scope}
            \begin{scope}[very thick, opacity=.5]
            \draw (0,0) -- ++(0:2.5) node[anchor=west]{\footnotesize $v_1=v_2$};
            \draw (0,0) -- ++(60:2.5) node[anchor=south]{\footnotesize $v_2=v_3$};
            \draw (0,0) -- ++(120:2.5);
            \draw (0,0) -- ++(180:2.5);
            \draw (0,0) -- ++(240:2.5);
            \draw (0,0) -- ++(300:2.5) node[anchor=west]{\footnotesize $v_1=v_3$};
            \end{scope}
            \begin{scope}
            \scriptsize
            \node (blue) at (2.5,1) {$v_1\leqslant v_2\leqslant v_3$};
            \node (orange) at (2.5,-1.5) {$v_2< v_1\leqslant v_3$};
            \node (red) at (-.3,2.5) {$v_1\leqslant v_3< v_2$};
            \node (magenta) at (-2.8,1.2) {$v_3< v_1\leqslant v_2$};
            \node (yellow!90!orange!70) at (-2.4,-.8) {$v_3< v_2 < v_1$};
            \node (green) at (0,-2.5) {$v_2\leqslant v_3 < v_1$};
            \end{scope}
        \end{tikzpicture}
        \caption{Let $w=3$. On the left, the integral points of the cube $9\square$ divided into its six chambers. On the right, a generic slice by a plane perpendicular to $(1,1,1)$, overlaid with the braid hyperplane arrangement.}
        \label{fig:chambers}
    \end{figure}
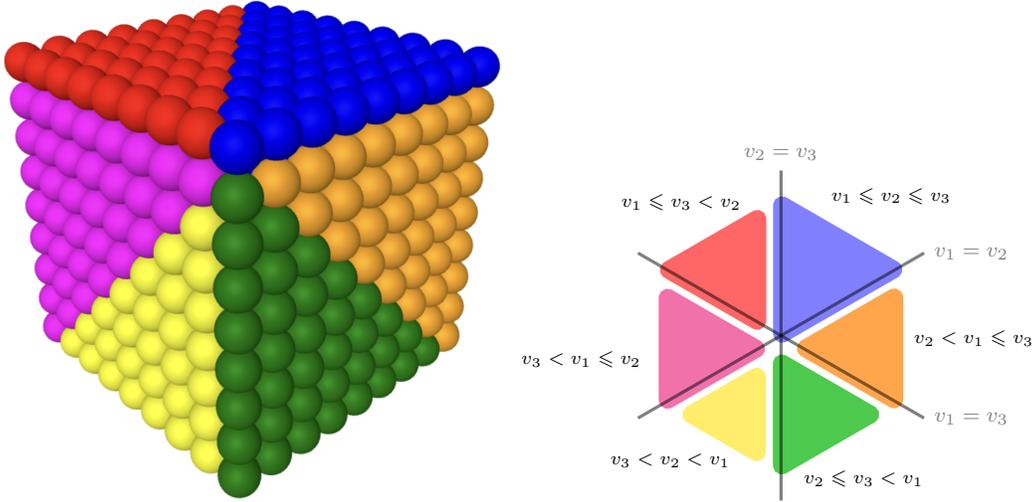
    
A finer subdivision can be obtained by breaking each region $\square\cap\{v_i \leqslant v_j\}$ into the sets of points $\square\cap\{v_i = v_j\}$ 
and $\square\cap\{v_i < v_j\}$. The resulting regions into which $\square$ is divided are known as \emph{faces} of the braid arrangement. The number of faces is $\sum_{\pi\in S_w} 2^{\asc(\pi)}$, and known 
as the $w$th \emph{Fubini number}. The faces are half-open polytopes naturally labeled by weak orders on $\{v_1,\ldots,v_w\}$ or, equivalently, by ordered set partitions 
$\gamma$ of $\{1,\ldots,w\}$, and denoted by $\finechamber(\gamma)$. See Figure \ref{fig:subdivisions} for an illustration of the faces of $\square$ 
for $w=3$.

\begin{example}\label{ex:fine chamber}
    Let $w=6$. Then
    \[
    \finechamber(23|145|6) = 
    \{v\in \square \mid v_2 = v_3 < v_1 = v_4 = v_5 < v_6\}.
    \]
\end{example}

An ordered set partition $\gamma$ gives rise to a permutation $\pi^\gamma\in S_w$ by ordering each part increasingly and concatenating 
the resulting words in order. The face $\finechamber(\gamma)$ is by construction contained in $\chamber(\pi^\gamma)$. 
\begin{example}
    We have $\pi^{23|145|6} = 231456$. The face $\finechamber(23|145|6)$ from Example~\ref{ex:fine chamber} is contained in the chamber 
    $\chamber(231456)$ from Example~\ref{ex:chamber}.
\end{example}

\begin{note}
    The faces of a hyperplane arrangement form a semigroup (see \cite{saliola2009face}).
    For each permutation $\pi$, let $F(\pi) = \finechamber(\pi(1)|\pi(2)|\ldots|\pi(w))$
    be the full-dimen\-sional face contained in $\chamber(\pi)$. Unraveling the definitions
    above, we get an alternative description of the chambers as
    \[
    \chamber(\pi) = \bigcup_{\finechamber(\gamma)F(12\ldots w) = F(\pi)} \finechamber(\gamma).
    \]
\end{note}

The sizes of the parts of an ordered set partition $\gamma$ form a composition $\alpha^\gamma$.

\begin{example}
    We have $\alpha^{23|145|6} = (2,3,1)$.
\end{example}

We show below that we can explicitly compute a `generating function' of the
points of a face $\finechamber(\gamma)$, and identify it with a certain
quasisymmetric polynomial labeled by $\alpha^\gamma$.
The notion of generating function that we need was introduced in \cite{BJR} for any matroid; we only need the special case in which the matroid comes from a hyperplane arrangement. Several aspects of this generating function are studied in \cite{GR.2014, BK, GPS}.

\begin{de}
The \emph{integer-point enumerator} of a $w$-dimensional polytope $P$ at height $h$ is the generating function of the integral points of $hP$, namely
\[
\ipe_P(h;x_0,x_1,\ldots) = \sum_{v\in hP\cap \ZZ^w} x_{v_1}x_{v_2}\cdots x_{v_w}.
\]
\end{de}
\begin{note}
    The specialization $\ipe_P(h;1,1,\ldots)$ is the Ehrhart polynomial $\ehr_P(h)$.
\end{note}
\begin{note}\label{note:finitely many variables}
The integer-point enumerator of a \emph{bounded} polytope is a polynomial in finitely many variables; in particular, the integer-point 
enumerator at height $h$ of any polytope contained in $\square$ is a polynomial in the variables $x_0, x_1, \ldots, x_h$.
\end{note}
\begin{lemma}
    The integer-point enumerator of $\finechamber(\gamma)$ at height $h$ is the monomial quasisymmetric polynomial $M_{\alpha^\gamma}(x_0,\ldots,x_h)$.
\end{lemma}
\begin{proof}
    Let $\alpha=\alpha^\gamma$, $\pi=\pi^\gamma$, and $P = \finechamber(\gamma)$. Suppose that $\alpha=(\alpha_1,\ldots,\alpha_k)$ has $k$ parts.
    We have
    \begin{align*}
        \ipe_P(h;x_0,x_1,\ldots) &= \sum_{v\in hP\cap \ZZ^w} x_{v_1}x_{v_2}\cdots x_{v_w}\\
        &= \sum_{v\in hP\cap \ZZ^w} x_{v_{\pi(1)}} x_{v_{\pi(2)}} \cdots x_{v_{\pi(w)}}\\
        &= \sum_{v\in hP\cap \ZZ^w} x_{v_{\pi(1)}}^{\alpha_1} x_{v_{\pi(\alpha_1+1)}}^{\alpha_2} \cdots x_{v_{\pi(\alpha_1+\cdots+\alpha_{k-1}+1)}}^{\alpha_k}\\
        &= \sum_{0\leqslant i_1 < \cdots < i_k \leqslant h} x_{i_1}^{\alpha_1} x_{i_2}^{\alpha_2} \cdots x_{i_k}^{\alpha_k}\\
        &= M_{\alpha}(x_0,\ldots,x_h). \qedhere
    \end{align*}
\end{proof}

Next, we compute the integer point enumerator of a chamber $\chamber(\pi)$.
Let $\pi\in S_w$ and consider the word $\pi(1)\pi(2)\cdots\pi(w)$. By recording the lengths of maximal increasing subsequences of the word, 
we obtain the \emph{descent composition} of $\pi$, that we denote by $\Des(\pi)$.
\begin{example}
    We compute $\Des(231456) = (2,4)$.
\end{example}

\begin{proposition}
    The integer-point enumerator of $\chamber(\pi)$ at height $h$ is the fundamental quasisymmetric polynomial $F_{\Des(\pi)}(x_0,\ldots,x_h)$.
\end{proposition}
\begin{proof}
    The chamber $\chamber(\pi)$ is the union of the faces $\finechamber(\gamma)$ for all $\gamma$ such that $\pi^\gamma = \pi$. In particular, note 
    that these set partitions are exactly those such that $\alpha^\gamma$ refines $\Des(\pi)$. Recall also from Note~\ref{note:finitely many variables} that the integer-point enumerator of $\finechamber$ at height $h$ is a polynomial in $x_0, \ldots, x_h$. Hence,
    \begin{align*}
    \ipe_{\chamber(\pi)}(h;x_0,x_1\ldots) &= \sum_{\gamma\ :\ \pi^\gamma=\pi} \ipe_{\finechamber(\gamma)}(h;x_0,x_1,\ldots)\\ 
    &= \sum_{\gamma\ :\ \pi^\gamma=\pi} M_{\alpha^\gamma}(x_0,\ldots,x_h)\\ 
    &= \sum_{\alpha\preccurlyeq \Des(\pi)} M_{\alpha}(x_0,\ldots,x_h)\\ 
    &= F_{\Des(\pi)}(x_0,\ldots,x_h).\qedhere
    \end{align*}
\end{proof}

Create a graph with vertex set $S_w$ by placing an edge $\{\pi,\tau\}$ whenever $\chamber(\pi^{-1})$ and $\chamber(\tau^{-1})$ are adjacent chambers (separated by one of the hyperplanes of the braid arrangement). We can turn this graph into a ranked poset, where the rank of $\pi$ is the minimal number of hyperplanes one has to cross to get from $\chamber(\pi^{-1})$ to the identity chamber $\chamber(12\ldots w)$. This is the \emph{weak order} \cite{BB}.

We now define a coarser subdivision of the cube. 
To define it as a real polytope, we need to extend RSK insertion to $\R^w$ in the obvious manner. Note that the insertion tableau of a point has real entries, whereas the recording tableau is a standard Young tableau. For any fixed standard Young tableau $Q$ of size $w$ we define the half-open polytope $\coarsechamber(Q)$ as
\[
\coarsechamber(Q) = \{
v \in \square \mid \RSK(v)_2 = Q
\}.
\]
Since all of the points inside a fixed $\chamber(\pi)$ have the same recording tableau,  this is truly a coarser subdivision.
See Figure \ref{fig:subdivisions}.
\begin{example}
\ytableausetup{smalltableaux}
    Let $w=6$.
    Recall from Example \ref{ex:chamber} that
    \[
    \chamber(231456) = \{v\in \square \mid v_2 \leqslant v_3 < v_1 \leqslant v_4 \leqslant v_5 \leqslant v_6\}.
    \]
    Perform RSK insertion of any point $v$ in the chamber. Start by inserting $v_1$ in the first row; then insert $v_2$, which bumps $v_1$ to the second row. The remaining entries all insert in the first row. We get
    \[
    \RSK(v) = \left(
    \ytableaushort{{v_2}{v_3}{v_4}{v_5}{v_6},{v_1}},~
    \ytableaushort{13456,2}
    \right).
    \]
\end{example} 
Each chamber in $w\square$
contains a unique point $v_\text{std}$ whose insertion tableau is standard. Let $u_\text{std}$ and $v_\text{std}$ be two such points.
They belong to adjacent chambers if they differ by a transposition of two consecutive entries $k$ and $k+1$ for some $k$. On the other hand, they differ by an elementary dual Knuth relation
\begin{align*}
\cdots (k+1) \cdots k \cdots (k+2) \cdots 
&\equiv
\cdots (k+2) \cdots k \cdots (k+1) \cdots 
\quad\text{or}\quad\\
\cdots k \cdots (k+2) \cdots (k+1) \cdots 
&\equiv
\cdots (k+1) \cdots (k+2) \cdots k \cdots 
\end{align*}
if and only if their recording tableaux are equal \cite[Thm.~3.6.10]{SaganBook}. Since dual Knuth relations arise as a transposition of consecutive entries, the above construction is truly a coarsening of the chamber subdivision.

\begin{figure}
    \centering
        \begin{tikzpicture}[x=2.5em,y=2.5em,rotate=30,baseline=0]
            \begin{scope}[line width=6, rounded corners=1pt, line cap = round]
            \filldraw[blue!50] (0,0) circle (2pt);
            \filldraw[blue!50] (0,0) ++ (.4,0) -- ++(0:1.6);
            \filldraw[blue!50, rotate=60] (0,0) ++ (.4,0) -- ++(0:1.6);
            \filldraw[blue!50] (.6,.34) -- ++(0:1.2) -- ++(120:1.2) -- cycle;
            \filldraw[red!60, rotate=60] (0,0) ++ (.6,.34) -- ++(0:1.2) -- ++(120:1.2) -- cycle;
            \filldraw[red!60, rotate=120] (0,0) ++ (.4,0) -- ++(0:1.6);
            \filldraw[magenta!70, rotate=120] (0,0) ++ (.6,.34) -- ++(0:1.2) -- ++(120:1.2) -- cycle;
            \filldraw[magenta!70, rotate=180] (0,0) ++ (.4,0) -- ++(0:1.6);
            \filldraw[orange!70, rotate=-60] (0,0) ++ (.6,.34) -- ++(0:1.2) -- ++(120:1.2) -- cycle;
            \filldraw[orange!70, rotate=-60] (0,0) ++ (.4,0) -- ++(0:1.6);
            \filldraw[green!70!black!70, rotate=-120] (0,0) ++ (.6,.34) -- ++(0:1.2) -- ++(120:1.2) -- cycle;
            \filldraw[green!70!black!70, rotate=-120] (0,0) ++ (.4,0) -- ++(0:1.6);
            \filldraw[yellow!90!orange!70] (-.6,-.34) -- ++(180:1.2) -- ++(-60:1.2) -- cycle;
            \end{scope}
            \begin{scope}[very thick, opacity=.5]
            \draw (0,0) -- ++(0:2.5);
            \draw (0,0) -- ++(60:2.5);
            \draw (0,0) -- ++(120:2.5);
            \draw (0,0) -- ++(180:2.5);
            \draw (0,0) -- ++(240:2.5);
            \draw (0,0) -- ++(300:2.5);
            \end{scope}
            \draw (0,0) ++ (1.2,.68) node[] {\tiny1|2|3};
            \draw[rotate=60] (0,0) ++ (1.2,.68) node[] {\tiny1|3|2};
            \draw[rotate=120] (0,0) ++ (1.2,.68) node[] {\tiny3|1|2};
            \draw[rotate=180] (0,0) ++ (1.2,.68) node[] {\tiny3|2|1};
            \draw[rotate=240] (0,0) ++ (1.2,.68) node[] {\tiny2|3|1};
            \draw[rotate=300] (0,0) ++ (1.2,.68) node[] {\tiny2|1|3};
            
            \draw (0,0) ++ (0:2.5) node[fill=white,circle,inner sep=.2] {\tiny12|3};
            \draw[rotate=60] (0,0) ++ (0:2.5) node[fill=white,circle,inner sep=.2] {\tiny1|23};
            \draw[rotate=120] (0,0) ++ (0:2.5) node[fill=white,circle,inner sep=.2] {\tiny13|2};
            \draw[rotate=180] (0,0) ++ (0:2.5) node[fill=white,circle,inner sep=.2] {\tiny3|12};
            \draw[rotate=240] (0,0) ++ (0:2.5) node[fill=white,circle,inner sep=.2] {\tiny23|1};
            \draw[rotate=300] (0,0) ++ (0:2.5) node[fill=white,circle,inner sep=.2] {\tiny2|13};
            
            \draw (0,0) node[fill=blue!50,circle,inner sep=0,scale=.8] {\tiny123};
        \end{tikzpicture}
        \hfill
        \begin{tikzpicture}[x=2.5em,y=2.5em,rotate=30,baseline=0]
            \begin{scope}[line width=6, rounded corners=1pt]
            \filldraw[blue!50] (0,0) -- ++(0:2) -- ++(120:2) -- cycle;
            \filldraw[red!60] (-.2,.34) -- ++(60:1.6) -- ++(180:1.6) -- cycle;
            \filldraw[orange!70] (.2,-.34) -- ++(-60:1.6) -- ++(60:1.6) -- cycle;
            \filldraw[magenta!70] (-.4,0) -- ++(180:1.6) -- ++(60:1.6) -- cycle;
            \filldraw[green!70!black!70] (-.2,-.34) -- ++(-60:1.6) -- ++(180:1.6) -- cycle;
            \filldraw[yellow!90!orange!70] (-.6,-.34) -- ++(180:1.2) -- ++(-60:1.2) -- cycle;
            \end{scope}
            \begin{scope}[very thick, opacity=.5]
            \draw (0,0) -- ++(0:2.5);
            \draw (0,0) -- ++(60:2.5);
            \draw (0,0) -- ++(120:2.5);
            \draw (0,0) -- ++(180:2.5);
            \draw (0,0) -- ++(240:2.5);
            \draw (0,0) -- ++(300:2.5);
            \end{scope}
            \draw (0,0) ++ (30:1.3) node[] {\footnotesize123};
            \draw[rotate=60] (0,0) ++ (30:1.3) node[] {\footnotesize132};
            \draw[rotate=120] (0,0) ++ (30:1.3) node[] {\footnotesize312};
            \draw[rotate=180] (0,0) ++ (30:1.3) node[] {\footnotesize321};
            \draw[rotate=240] (0,0) ++ (30:1.3) node[] {\footnotesize231};
            \draw[rotate=300] (0,0) ++ (30:1.3) node[] {\footnotesize213};
        \end{tikzpicture}
        \hfill
        \begin{tikzpicture}[x=2.5em,y=2.5em,rotate=30,baseline=0]
        \ytableausetup{smalltableaux}
            \begin{scope}[line width=6, rounded corners=1pt]
            \filldraw[blue!50] (0,0) -- ++(0:2) -- ++(120:2) -- cycle;!70!70
            \filldraw[red!50!magenta!50] (-.2,.34) -- ++(60:1.6) -- ++(180:1.6) -- ++(240:2) -- ++(0:1.6) -- cycle;
            \filldraw[orange!60!green!70] (.2,-.34) -- ++(0:1.6) -- ++(240:1.6) -- ++(180:2) -- ++(60:1.6) -- cycle;
            \filldraw[yellow!90!orange!70] (-.6,-.34) -- ++(180:1.2) -- ++(-60:1.2) -- cycle;
            \end{scope}
            \begin{scope}[very thick, opacity=.5]
            \draw (0,0) -- ++(0:2.5);
            \draw (0,0) -- ++(60:2.5);
            \draw (0,0) -- ++(120:2.5);
            \draw (0,0) -- ++(180:2.5);
            \draw (0,0) -- ++(240:2.5);
            \draw (0,0) -- ++(300:2.5);
            \end{scope}
            \node () at (1.2,.68) {\ytableaushort[*(white)]{123}};
            \node () at (-.8,1) {\ytableaushort[*(white)]{12,3}};
            \node () at (.4,-1) {\ytableaushort[*(white)]{13,2}};
            \node () at (-1.4,-.9) {\ytableaushort[*(white)]{1,2,3}};
        \end{tikzpicture}
    \caption{On the left, the $\finechamber(-)$ faces of $\square$ for $w=3$; in the middle, the $\chamber(-)$ chambers; on the right, the $\coarsechamber(-)$ subdivision.}
    \label{fig:subdivisions}
\end{figure}
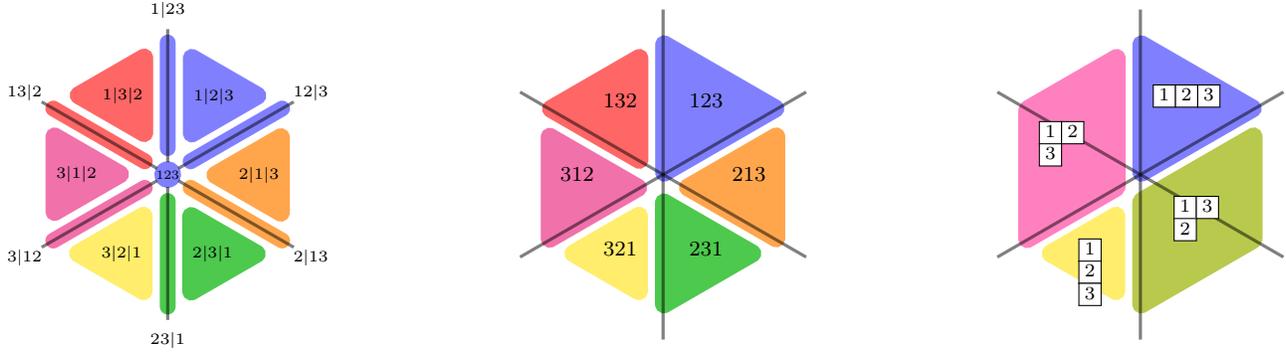

Note also that since mapping $\pi$ to $\pi^{-1}$ interchanges the insertion and recording tableaux under RSK, we can write
\begin{equation}
\label{equation.Ch pi}
	\coarsechamber(Q) = \bigcup_{\substack{\pi \in S_w\\ \RSK(\pi)_1 = Q}} \chamber(\pi).
\end{equation}

While the interpretation of the faces of the Coxeter
complex are the integer point enumerators
of the quasi-symmetric functions \cite{BK,GPS,BJR},
here we glue faces together
to obtain Schur polynomials as integer point
enumerators of pieces of the Coxeter complex.

\begin{theorem}\label{theorem:integer_point_enumerator}
    Let $Q\in\SYT(\mu)$.
    The integer-point enumerator of $\coarsechamber(Q)$ at height $h$ is the Schur polynomial $s_{\mu}(x_0,\ldots,x_h)$.
\end{theorem}
\begin{proof} 
Fix $Q\in\SYT(\mu)$ and let $P=\coarsechamber(Q)$.
Recall that RSK gives a bijection between the integer points of $h\square$ and $\bigcup_{\lambda\vdash w} \SSYT_{h+1}(\lambda)\times\SYT(\lambda)$, where 
the first tableau is the insertion tableau and the second one is the recording tableau.
Given $v \in \chamber(\pi)$, let $x_v = x_{v_1}\cdots x_{v_w}$. Note that if $\RSK(v) = (T,Q)$ then $x_v = x_T$.
Since $hP\cap\ZZ^w$ is the preimage under RSK of $\SSYT_{h+1}(\mu)\times\{Q\}$
we can write
\[
\ipe_{P}(h;x_0,x_1,\ldots) = \sum_{v\in hP\cap \ZZ^w} x_{v} = \sum_{T\in\SSYT_{h+1}(\mu)} x_{T}.
\]
Hence
$\ipe_P(h;x_0,x_1,\ldots) = s_\mu(x_0,\ldots,x_h)$ by the combinatorial definition of Schur polynomials and Note~\ref{note:finitely many variables}.
\end{proof}

\begin{example}
    Let $w=4$ and see Figure \ref{fig:4d} for the chambers of (the front part of) $\square$. We can see $5$ of the coarse regions in the illustration. 
    Three coarse regions $\chamber(3124)\sqcup\chamber(1324)\sqcup\chamber(1324)$, 
    $\chamber(1432)\sqcup\chamber(4132)\sqcup\chamber(4312)$, 
    and $\chamber(1243)\sqcup\chamber(1423)\sqcup\chamber(1342)$ are composed of three chambers each, one coarse region 
    $\chamber(3142)\sqcup\chamber(3412)$ is composed of two chambers, and the remaining coarse region is $\chamber(1234)$.
\end{example}

\begin{figure}
    \centering
    \includegraphics[width=0.5\linewidth]{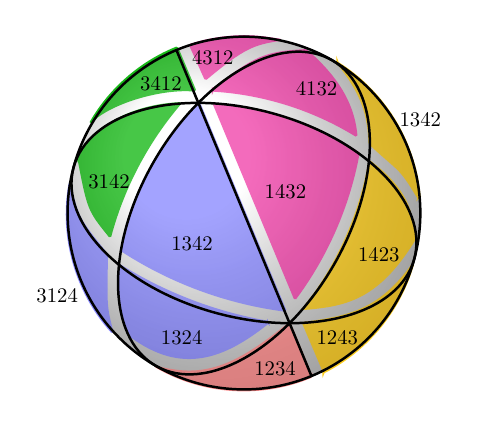}
    \caption{A generic slice of $[0,1]^4$ projected to a $3$-dimensional sphere. Illustrated is the front part of the sphere, containing $12$ chambers, 
    indexed by permutations of $S_4$. 
    }
    \label{fig:4d}
\end{figure}

\subsection{Quantum Ehrhart polynomials}

We now study the \emph{integer-point enumerator series}
$$\Ipe_P(z;x_0,x_1,\ldots) = 
\sum_{h\geqslant 0}\ipe_{P}(h;x_0,x_1, \ldots) z^h,$$
where $P$ is a polytope of the form $\finechamber(\gamma)$, $\chamber(\pi)$, or $\coarsechamber(Q)$.
In order to do this, it is convenient to consider the cone $\bigcup_{h\geqslant1} h\square = \R_{\geqslant1}^w$.
Define the \emph{principal chamber} $\triangle = \chamber(12\ldots w)$. For any chamber $\chamber(\pi)$, let 
$\chamber(\pi)_{\geqslant 1} = \bigcup_{h\geqslant 1} h\chamber(\pi)$. Explicitly, we can write 
$$\triangle_{\geqslant 1}\cap\ZZ^w = \mathbb{Z}_{\geqslant 0}(0,\ldots,0,1) + \mathbb{Z}_{\geqslant 0}(0,\ldots,0,1,1) 
+ \cdots + \mathbb{Z}_{\geqslant 0}(1,1,\ldots,1).$$
Any other $\chamber(\pi)_{\geqslant 1}$ is obtained by reflections and translations of $\triangle_{\geqslant 1}$.

\begin{lemma}\label{lem:chambers in terms of pral}
For any $\pi\in S_w$, we have
\[
\chamber(\pi)_{\geqslant 1}\cap\ZZ^w = p + \mathbb{Z}_{\geqslant 0}\pi(0,\ldots,0,1) + \mathbb{Z}_{\geqslant 0}\pi(0,\ldots,0,1,1) + \cdots 
+ \mathbb{Z}_{\geqslant 0}\pi(1,1,\ldots,1),
\]
where $p$ is the vector given by
\[
p_{\pi(1)} = 0 
~~\text{and}~~
p_{\pi(i+1)} = 
\begin{cases}
    p_{\pi(i)}   &\text{if }\pi(i) < \pi(i+1),\\
    p_{\pi(i)}+1 &\text{if }\pi(i) > \pi(i+1).
\end{cases}
\]
\end{lemma}
\begin{proof}
We can obtain the Weyl chamber $\overline{\chamber(\pi)}$ from the principal Weyl chamber $\triangle$ by successive reflections with respect to the hyperplanes of the braid arrangement. The vector $p$ from the statement of the lemma is the smallest possible integer vector in $\chamber(\pi)_{\geqslant 1}$ by construction.
\end{proof}

If $P$ is a $w$-dimensional polytope, $\ehr_P(h) = \#(hP\cap\ZZ^w)$ and the \emph{Ehrhart series} is $\Ehr_P(z) = \sum_{h\geqslant 0} \ehr_P(h)z^h$.

For the purpose of studying plethysm coefficients, it is useful to introduce a $q$-analogue of Ehrhart theory, first considered by Chapoton 
in \cite{Chapoton}. Let $g\in\R^w$. The \emph{$q$-Ehrhart polynomial} of $P$ with respect to the grading $g$ is
\[
\qehr^g_P(h,q) = \sum_{v\in hP\cap\ZZ^w} q^{\langle g,v\rangle},
\]  
and the \emph{$q$-Ehrhart series} of $P$ is $\qEhr^g_P(z,q) = \sum_{h\geqslant 0}\qehr_P^g(h,q)z^h$. 
The $q$-Ehrhart theory is developed further in \cite{BK-Chapoton}; see also \cite{RR24, DLVVW.2024} for other related generalizations.

\begin{proposition}
    Let $u = (1,1,\ldots,1)$. We have
    $\qEhr^u_P(z,q) = \Ipe_P(z;1,q,q^2,\ldots).$
\end{proposition}
\begin{proof}
    Follows straight from the definitions.
\end{proof}
To apply the theory to characters of $\SL_2$, which are using centered Laurent polynomials in $q$, we introduce the following definitions.

\begin{de}
\label{definition.P tilde}
    Let $u = (1,\ldots,1)\in\ZZ^w$.
    Given $P$ a $w$-dimensional polytope, let $\tilde{P} = P-\frac{1}{2}u$.
    Define the \emph{quantum Ehrhart polynomial} of $P$ as 
    $$\Qehr_P(h,q) = \sum_{v\in h\tilde{P}\cap\mathcal{L}_h} q^{2\langle u,v\rangle}$$
    and the \emph{quantum Ehrhart series} of $P$ as $\QEhr_P(z,q) = \sum_{h\geqslant 0} \Qehr_P(h,q)z^h$, where $\mathcal{L}_h$ is the 
    lattice $(\frac{h}{2}+\ZZ)^w$.
\end{de}
\begin{lemma}\label{lem:from qEhr to QEhr}
    Let $P$ be a $w$-dimensional polytope contained in $\square$. Then,
    $\Qehr_P(h,q) = \qehr_P^u(h,q^2)/q^{wh}$ and $\QEhr_P(z,q) = \qEhr_P^u\big(\frac{z}{q^w},q^2\big)$.
\end{lemma}
\begin{proof}
    We compute directly
    \begin{align*}
        \Qehr_P(h,q) &= \sum_{v\in h\tilde{P}\cap(\frac{h}{2}+\ZZ)^w} q^{2\langle u,v\rangle} 
        = \sum_{v\in hP\cap\ZZ^w} q^{2\langle u,v-\frac{h}{2}u\rangle}\\
        &= \sum_{v\in hP\cap\ZZ^w} q^{2\langle u,v\rangle-wh}
        = q^{-wh} \qehr_P^u(h,q^2),
    \end{align*}
    from which the second result follows.
\end{proof}

Collecting results from above, the quantum Ehrhart series of $\coarsechamber(Q)$ is a generating function for plethysm of Schur polynomials.
Indeed, let $\mu = \shape(Q)$ be a partition of $w$ and recall Note~\ref{note:finitely many variables} to obtain
\begin{align*}
    \QEhr_{\coarsechamber(Q)}(z,q)  &=\Ipe_{\coarsechamber(Q)}(z/q^w;1,q^2,q^4,\ldots) \\&=
    \sum_{h\geqslant 0}\ipe_{\coarsechamber(Q)}(h;1,q^2,\ldots)z^h/q^{wh}\\&=
    \sum_{h\geqslant 0}s_\mu(1,q^2,\ldots,q^{2h})z^h/q^{wh}\\&=
    \sum_{h\geqslant 0}s_\mu(q^{-h},q^{2-h},\ldots,q^{h}) z^h\\&=
    \sum_{h\geqslant 0}s_\mu[s_h](q^{-1},q) z^h.
\end{align*}
When $\mu=(w)$, this generating function is well-understood.

\begin{theorem}[Heine's formula]
    Let $Q = \ytableaushort{12{\cdots}w}$ so that $\coarsechamber(Q) = \triangle$. Then, 
    $$\QEhr_\triangle(z,q) = \sum_{h\geqslant 0} \Qbinom{w+h}{w} z^h = \prod_{i=0}^w \frac{1}{1-q^{w-2i}z}.$$
\end{theorem}
\begin{proof}
     It follows from \cite{Heine} after a straightforward reparametrization (see also \cite{QuantumCalculus}, \cite[\S2.3]{BK-Chapoton}).
\end{proof}

For any other partition $\mu$ of $w$, we can compute the quantum Ehrhart series explicitly. 
The right hand equality in the next theorem appears in the literature in~\cite{GOS},~\cite[Thm.~5.3]{GesselReutenauer1993}, 
and~\cite[\S2]{ShareshianWachs2010}. By relating it to the quantum Ehrhart series, we give a new geometric perspective to this identity.

\begin{theorem}\label{th:grsw_gf}
    Let $Q\in\SYT(\mu)$ and let $P = \coarsechamber(Q)$.  Then 
    $$\QEhr_P(z,q) = \sum_{h\geqslant 0} s_\mu[s_h](q^{-1},q) z^h =  \frac{\sum_{T\in\SYT(\mu)} q^{-2\maj(T)} (zq^w)^{\des(T)}}{\prod_{i=0}^w(1-q^{w-2i}z)}.$$
\end{theorem}

\begin{remark}
\label{remark.QEhr partition}
Since the right hand side of the expression in Theorem \ref{th:grsw_gf} is independent
of the choice of standard tableaux $Q$, we may define
$\QEhr_\mu(z,q) := \QEhr_{\coarsechamber(Q)}(z,q)$ for any $Q \in \SYT(\mu)$.
\end{remark}
\begin{proof}
    The polytope $P$ is a union of chambers $\bigcup \chamber(\pi)$ indexed by some permutations $\pi$.
    For any choice of $\pi$,
    Lemma \ref{lem:chambers in terms of pral} describes a point $p$ such that $h\chamber(\pi)\cap\ZZ^w = p + h'\cdot\pi(\triangle)$ for some $h'$.
    The precise value of $h'$ can be computed as $h - \max_i p_i = h-\des(\pi)$ by construction of $p$. Since $\ipe_\triangle(h;x_0,x_1,\ldots)$ is the symmetric polynomial $s_{(w)}(x_0,\ldots,x_h)$ and for any polytope $R$
    \[
    \QEhr_R(z,q) =
    \qEhr_R^u(z/q^w,q^2) =
    \sum_{h\ge0}\ipe_R(h;1,q^2,q^4,\ldots){z^h}/{q^{wh}},
    \]
    we deduce $\QEhr_{\pi R} = \QEhr_R$. Hence
    \begin{align*}
        \QEhr_{\chamber(\pi)}^u(z,q)
        &=
        (z/q^w)^{\des(\pi)} q^{2\langle u,p\rangle} \QEhr_{\triangle}^u(z,q).
    \end{align*}
    By direct computation,
    \[
    \langle u,p\rangle =
    \sum_{i\in\Des(\pi)} (w-i) = 
    w\cdot\des(\pi) - \sum_{i\in\Des(\pi)} i = 
    w\cdot\des(\pi) - \maj(\pi).
    \]
    Therefore, \[
    (z/q^w)^{\des(\pi)} q^{2\langle u,p\rangle}= 
    z^{\des(\pi)}
    q^{w\cdot\des(\pi) - 2\maj(\pi)}.
    \]
    
    Recall that each chamber $\chamber(\pi)_{\geqslant 1}$ contains exactly one \emph{standard point} $v_\text{std}$, characterized by the 
    property that its insertion tableau $T(v_\text{std})$ is standard.
    Note that $\maj(T) = \maj(\pi)$, and $\des(T)=\des(\pi)$.
    Moreover, since $\SYT(\mu)$ is contained in $\SSYT_{[0,h]}(\mu)$ for $h\geqslant w$, every standard tableau arises this way exactly once. 
    Summing over all chambers that compose $P$, we obtain
    \begin{align*}
        \QEhr_P(z,q) &= \QEhr_\triangle(z,q) \sum_{\pi}q^{-2\maj(\pi)} (zq^w)^{\des(\pi)}\\
        &= \QEhr_\triangle(z,q) \sum_{\substack{v_\text{std}\in P\\\text{standard}}}q^{-2\maj(T(v_\text{std}))} (zq^w)^{\des(T(v_\text{std}))}\\
        &= \QEhr_\triangle(z,q) \sum_{T\in\SYT(\mu)}q^{-2\maj(T)} (zq^w)^{\des(T)}. \qedhere
    \end{align*}
\end{proof}
Note that we can reprove the Carlitz identity \cite{Carlitz} (already found in \cite[II.IV, \S462]{MacMahon}; see also \cite[\S2.1]{BK-Chapoton}). 
We state it with the quantum specialization.
\begin{theorem}[Carlitz identity]
    We have 
    $$\QEhr_\square(z,q) =  \frac{\sum_{\pi\in S_w} q^{-2\maj(\pi)} (zq^w)^{\des(\pi)}}{\prod_{i=0}^w(1-q^{w-2i}z)}.$$
\end{theorem}
\begin{proof}
    It follows from 
    $$\square = \bigcup_{\substack{\mu\vdash w\\ Q\in\SYT(\mu)}} \coarsechamber(Q) = \bigcup_{\pi\in S_w} \chamber(\pi)$$
    and the previous theorem.
\end{proof}

\subsection{MacMahon's combinatory analysis}
\label{ss:MacMahon}

MacMahon's partition analysis~\cite{MacMahon} was introduced in 1916 in the context of partitions
and turns out to be of great importance for the analysis of generating functions, see for 
example~\cite{APR-Omega.2001, APR.2001, Xin.2004, BZ-Omega}.
In this section, we show that the generating function $A_\mu(z,q)$ of $\SL_2$-plethysm coefficients defined
in~\eqref{equation.Amu} is related to the quantum Ehrhart series via the positive term operator.

MacMahon's Omega operator $\Omega^{(\alpha_1,\ldots,\alpha_r)}_{\geqslant}=\Omega_{\geqslant}$ is defined as 
\[
	\Omega^{(\alpha_1,\ldots,\alpha_r)}_{\geqslant}\sum_{s_1=-\infty}^\infty \cdots \sum_{s_r=-\infty}^\infty 
	A_{s_1,\ldots,s_r} \alpha_1^{s_1} \cdots \alpha_r^{s_r}
	= \sum_{s_1=0}^\infty \cdots \sum_{s_r=0}^\infty A_{s_1,\ldots,s_r},
\]
where the domain of the $A_{s_1,\ldots,s_r}$ is the field of rational functions over $\mathbb{C}$ in several complex
variables and the variables $\alpha_i$ are restricted to a neighborhood of the circle $|\alpha_i|=1$. In~\cite{Xin.2004}, Xin studied
the positive term operator
\begin{equation}
\label{equation.PT}
	\mathsf{PT}^\alpha\sum_{n=-\infty}^\infty a_n \alpha^n  = \sum_{n=0}^\infty a_n \alpha^n
\end{equation}
and more generally $\mathsf{PT}^{(\alpha_1,\ldots,\alpha_r)}$ by iteration of~\eqref{equation.PT}. He showed that~\cite[Equation (4.1)]{Xin.2004}
\[
	\Omega_{\geqslant} f(\alpha_1,\ldots,\alpha_r,x) 
	= \mathsf{PT}^{(\alpha_1,\ldots,\alpha_r)} f(\alpha_1,\ldots,\alpha_r,x)|_{(\alpha_1,\ldots,\alpha_r)
	=(1,\ldots,1)}.
\]
An \emph{Elliott-rational function} is a rational function that can be written in such
a way that its denominator can be factored into the products of one monomial minus another, with the 0 monomial allowed.
See for example~\cite[Definition 3.1]{Xin.2004}.

\begin{thm}\cite[Remark 4.9, Theorem 3.2]{Xin.2004}
\label{thm:Elliott-rational}
If $f$ is Elliott-rational, then $\mathsf{PT}^\alpha f$ is still Elliott-rational.
\end{thm}

Recall the quantum Ehrhart series indexed by a partition as in Remark~\ref{remark.QEhr partition}.
The main result of this section is the following.
\begin{theorem}
\label{theorem.A QEhr}
For $\mu$ a partition, we have
\begin{equation}
	A_\mu(z,q) = \mathsf{PT}^q (q-q^{-1})\QEhr_\mu(z,q).
\end{equation}
\end{theorem}

\begin{proof}
We defined $a_{\mu[h]}^{[k]}$ to be the coefficient of $[k]_q$ in the plethysm $s_\mu[s_h](q,q^{-1})$ for $k\geqslant 1$;
recall that by definition $a_{\mu[h]}^{[0]}=0$. This is equivalent to
\[
	a_{\mu[h]}^{[k]} = (q-q^{-1}) s_\mu[s_h](q,q^{-1})\big|_{q^k} \qquad \text{for $k \geqslant 0$.}
\]
Hence, using the definition of $A_\mu(z,q)$ in~\eqref{equation.Amu}, we obtain
\[
	A_\mu(z,q) = \sum_{k\geqslant 0, h\geqslant 0} a_{\mu[h]}^{[k]} q^k z^h
	= \sum_{h\geqslant 0} \mathsf{PT}^q (q-q^{-1}) s_\mu[s_h](q,q^{-1}) z^h
	=\mathsf{PT}^q (q-q^{-1})\QEhr_\mu(z,q),
\]
where the last equation follows from Theorem~\ref{th:grsw_gf}.
\end{proof}

\begin{cor}
\label{cor.A rational}
For a partition $\mu$, $A_\mu(z,q)$ is Elliott-rational.
\end{cor}

\begin{proof}
By Theorem~\ref{th:grsw_gf}, $\QEhr_\mu(z,q)$ is Elliott-rational in both $q$ and $z$. It then follows from
Theorems~\ref{theorem.A QEhr} and~\ref{thm:Elliott-rational} that $A_\mu(z,q)$ is Elliott-rational.
\end{proof}

We give two more conceptual ways of understanding the $\PT$ operator.
Following \cite{BZ-Omega}, the $\PT$ operator of a $q$-Ehrhart series is
\[
\PT[q] \qEhr^g_P(z,q) = \qEhr^g_{P\cap H^g}(z,q),
\]
where $H^g = \{v\in\R^w \mid \langle g,v\rangle \geqslant 0\}$. This is clear from the definition of the $q$-Ehrhart series, but gives a geometric interpretation of the operator. 
See Figure~\ref{fig:Omega}.

\begin{figure}
   \centering
    \begin{tikzpicture}[x=2em,y=2em,z=-1em]
        \fill[blue!30, opacity=.2] (0,0,3) -- (2,0,1.5) -- (2,3,1.5) -- (0,3,3) -- cycle;
        \fill[blue!30, opacity=.2] (2,0,1.5) -- (2,3,1.5) -- (5,3,3) -- (5,0,3) -- cycle;
        \filldraw[black, fill = red!50, very thick, opacity=.6, dotted] (1,3,2.25) -- (1.5,3,3.75) -- (3,1.5,4.5) -- (4,0,3.75) -- (3.5,0,2.25) -- (2,1.5,1.5) -- cycle;
        \draw[thick, opacity=.6, dotted] 
            (2,1.5,1.5) -- (2,3,1.5)
            (3.5,0,2.25) -- (5,0,3);
        \fill[blue!50, opacity=.2] (0,0,3) -- (3,0,4.5) -- (3,3,4.5) -- (0,3,3) -- cycle;
        \fill[blue!30, opacity=.2] (3,0,4.5) -- (3,3,4.5) -- (5,3,3) -- (5,0,3) -- cycle;
        \fill[blue!10, opacity=.2] (0,3,3) -- (3,3,4.5) -- (5,3,3) -- (2,3,1.5) -- cycle;
        \filldraw[black, fill = blue!40, very thick, opacity=.6] (3,1.5,4.5) -- (3,3,4.5) -- (1.5,3,3.75);
        \filldraw[black, fill = blue!30, very thick, opacity=.6] (3,1.5,4.5) -- (3,3,4.5) -- (5,3,3) -- (5,0,3) -- (4,0,3.75) -- cycle;
        \filldraw[black, fill = blue!20, very thick, opacity=.6] (1,3,2.25) -- (1.5,3,3.75) -- (3,3,4.5) -- (5,3,3) -- (2,3,1.5) -- cycle;
        \draw[black, thick] (1,3,2.25) -- (1.5,3,3.75) -- (3,1.5,4.5) -- (4,0,3.75);
    \end{tikzpicture}
   \caption{Let $\tilde{P} = [-1/2,1/2]^3$. The $\PT[q]$ operator applied to $\qEhr^u_{\tilde{P}}(z,q)$ with
   $u = (1,\ldots,1)$ is the $q$-Ehrhart series of the displayed sub-polytope of $\tilde{P}$, with respect to the grading $u$.}
   \label{fig:Omega}
\end{figure}
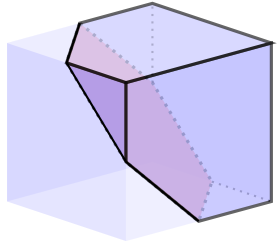

To apply this to the quantum setting, we need some extra care. Using Lemma~\ref{lem:from qEhr to QEhr}, we compute
\[
\PT[q] \QEhr_P(z,q) =
\PT[q] \qEhr_P^u(z/q^w,q^2) =
\sum_{h\geqslant0} \sum_{\substack{v\in hP\cap\ZZ^w \\ 2\langle u, v\rangle\geqslant wh}} q^{2\langle u,v\rangle - wh}z^h.
\]
Note that the right-hand side is not in the form of a quantum Ehrhart series of a polytope.
\medskip

The second conceptual way of understanding the $\PT$ operator is through the lens of quantum calculus. 
The relationship between $\QEhr_\mu(z,q)$ and $A_\mu(z,q)$ is evident from the following formulas:
\[
\QEhr_\mu(z,q) =  \sum_{k\geqslant 1, h\geqslant 0} a_{\mu[h]}^{[k]} [k]_q z^h
\quad\text{and}\quad
A_\mu(z,q) = \sum_{k\geqslant 1, h\geqslant 0} a_{\mu[h]}^{[k]} q^k z^h.
\]
To turn $A_\mu(z,q)$ into $\QEhr_\mu(z,q)$, one needs to construct a map sending $q^k$ to $[k]$. This is the map
\[
\frac{f(q)-f(q^{-1})}{q-q^{-1}}
\longmapsfrom
f(q)
\]
which is a specialization of the $q$-derivative defined in \cite[Ch.~26]{QuantumCalculus}.
Hence the map 
\[
f(q) \longmapsto \PT[q] (q-q^{-1}) f(q)
\]
that sends $\QEhr_\mu(z,q)$ to $A_\mu(z,q)$ is its $q$-antiderivative \cite[Ch.~19]{QuantumCalculus}.

\section{Combinatorial interpretation for plethysm coefficients}
\label{section.combinatorial}

In this section, we show how the bivariate generating function $A_\mu(z,q)$ relates to combinatorial statements
about $\SL_2$-plethysm coefficients in Section~\ref{ss:GeneratingFunctions}. In addition, we give
linear recursions that completely specify the $\SL_2$-plethysm coefficients in Section~\ref{ss:recursion}.

\subsection{Generating functions, cones, and partitions}
\label{ss:GeneratingFunctions}

In Theorem~\ref{theorem.A QEhr}, we showed that 
\[
	A_\mu(z,q) := \sum_{h\geqslant 0} \sum_{k\geqslant 0} a_{\mu[h]}^{[k]} q^k z^h 
	= \mathsf{PT}^q (q-q^{-1})\QEhr_{\mu}(z,q)
\]
and concluded in Conjecture~\ref{cor.A rational} that $A_\mu(z,q)$ is rational.
In this section, we give explicit rational expressions for $A_\mu(z,q)$ with positive numerators when $|\mu|\leqslant 5$,
illustrating how these generating functions can be converted to a combinatorial interpretation for the coefficients 
$a_{\mu[h]}^{[k]}$ as a union of polyhedral cones. 

\begin{table}
\begin{align*}
    A_{1}(z,q) & = \frac{q}{(1 - qz)} \\
    A_{2}(z,q) & = \frac{q}{(1-z^2)(1-q^2z)} \\
    A_{11}(z,q) & = \frac{qz}{(1 - z^2) (1- q^2z)} \\
    A_3(z,q) &= \frac{q+q^4z^3}{(1-z^4)(1-q^2z^2)(1-q^3z)}\\
    A_{21}(z,q) &= \frac{q^2z}{(1-z^2)(1-qz)(1-q^3z)}\\
    A_{111}(z,q) &= \frac{qz^2+q^4z^5}{(1-z^4)(1-q^2z^2)(1-q^3z)}\\
    A_{4}(z,q) & = \frac{q+q^7z^3}{(1-z^3)(1-z^2) (1-q^4 z^2) (1-q^{4} z)} \\
    A_{31}(z,q) & = \frac{q^3z}{(1-z^2) (1-z) (1-q^{2} z) (1-q^{4} z) } \\
    A_{22}(z,q) & = \frac{ qz +q^7 z^{4}}{(1-z^3) (1-q^4 z^2) (1-z) (1-q^4 z)} \\
    A_{211}(z,q) & = \frac{q^{3}z^2}{(1-z^2) (1-z) (1-q^{2} z) (1-q^{4} z) } \\
    A_{1111}(z,q) & = \frac{q z^3+q^7 z^{6}}{(1-z^3) (1-z^2) (1-q^{4} z^2) (1-q^4 z)} \\
\end{align*}
\caption{Rational expressions for $A_\mu(z,q)$ for $|\mu|\leqslant 4$.}
\label{table.rational A}
\end{table}

The positive rational expressions for $A_\mu(z,q)$ for $|\mu|\leqslant 4$ are given in Table~\ref{table.rational A}, which
were computed in {\sc SageMath}~\cite{sagemath}
from the explicit expression in Theorem~\ref{th:grsw_gf} by applying the $\mathsf{PT}^q$
operator.  We then put all of these generating functions in a very similar and suggestive form such that the numerator
is a polynomial with positive coefficients.
The generating functions $A_\mu(z,q)$ for some select partitions $|\mu|\leqslant 5$ are given in Appendix~\ref{appendix.A} 
and have a surprising jump in complexity from all of generating functions $A_\mu(z,q)$ for
$|\mu| \leqslant 4$. 

Each summand in these generating functions can be interpreted
as the generating function of points in a polyhedral cone.
For example, $A_3(z,q)$ is a sum of two
monomials over a product of terms of the form $1-q^a z^b$.

The summand
\begin{equation} \label{eq:examp_one}
	\frac{q}{(1-z^4)(1-q^2z^2)(1-q^3z)}
\end{equation}
is the generating function $\sum_{h,k \geqslant 0} |B^{(1)}_{h,k}| z^h q^k$
where $B^{(1)}_{h,k}$ is a set of points in $\mathbb{Z}^4$. The first coordinate represents
the weight in $z$ and the other three represent the weight in $q$.
Let $\vec{u} = (4,0,0,0)$ represent the weight of $z^4$, $\vec{v} = (2,0,2,0)$
represent the weight $q^2 z^2$ and $\vec{w} = (1,0,0,3)$ the weight of $q^3z$.  Then the set
\begin{equation}\label{set1}
	B^{(1)}_{h,k} = \{ \vec{p} = a \vec{u} + b \vec{v} + c \vec{w} \mid a,b,c \geqslant 0, p_1 = h, p_2+p_3+p_4 = k-1\}
\end{equation}
has the generating function in Equation \eqref{eq:examp_one} while the set
\begin{equation}\label{set2}
	B^{(2)}_{h,k} = \{  \vec{p} = a \vec{u} + b \vec{v} + c \vec{w} \mid a,b,c \geqslant 0, p_1 = h-3, p_2+p_3+p_4 = k-4 \}
\end{equation}
has as generating function $\sum_{h,k \geqslant 0} |B^{(2)}_{h,k}| z^h q^k$
for the second summand of $A_3(z,q)$
\[
	\frac{q^4z^3}{(1-z^4)(1-q^2z^2)(1-q^3z)}~.
\]
Therefore, we can say that a combinatorial interpretation for $a_{3[h]}^{[k]}$
is the number of points in the (disjoint) union of these two sets.

The polyhedral cone might seem like a contrived combinatorial interpretation,
however if another combinatorial interpretation is found
it could then potentially be proven via a bijection with the points in the cone.
We demonstrate this with the following interpretation.

\begin{proposition}[{\cite[Thm.~5.3]{COSSZ.2022}}]
Define the set
$$\Pi(3) =
\left\{(\mu_1,\mu_2) \text{ partition} \mid \mu_2 \text{ even and } \mu_1 \gge 2\mu_2\right\} ~ ,$$
where $\mu_1 \gge 2\mu_2$ means that $\mu_1 \geqslant 2\mu_2$ but $\mu_1 \neq 2\mu_2+1$.
Then $a_{3[h]}^{[k]}$ is equal to the number of $(\mu_1,\mu_2) \in \Pi(3)$ such that
$\mu_1 \leqslant h$ and $3h-2(\mu_1 + \mu_2)+1 = k$.
\end{proposition}

\begin{proof}
Fix integers $h$ and $k$ and let $(\mu_1, \mu_2)$ be a point in $\Pi(3)$ such that
$\mu_1\leqslant h$ and $3h-2(\mu_1+\mu_2)+1=k$. Then we have:
\begin{itemize}
\item If $\mu_1$ is even, then $\vec{p} = (h,0,\mu_1-2\mu_2,k-1-\mu_1+2\mu_2)$
can be shown to be a point in $B^{(1)}_{h,k}$
and the inverse transformation takes the point
$\vec{p}$ and sends it to $(\mu_1,\mu_2) = (\frac{3p_1-p_4}{3},\frac{3p_1-3p_3-p_4}{6} )$.
\item If $\mu_1$ is odd, then $\vec{p} = (h-3,0,\mu_1-2\mu_2-3,k-\mu_1+2\mu_2-1)$ is a point in $B^{(2)}_{h,k}$.
The inverse transformation which takes a point of $B^{(2)}_{h,k}$ sends $\vec{p}$ to the point
$(\mu_1,\mu_2) = (\frac{3p_1-p_4+9}{3}, \frac{3p_1-3p_3-p_4}{6})$.
\end{itemize}
Note that the conditions on the weights in $B^{(1)}_{h,k}$ and $B^{(2)}_{h,k}$ follow from the conditions
$\mu_1\leqslant h$ and $3h-2(\mu_1+\mu_2)+1=k$.
\end{proof}

Similar bijections recover the combinatorial interpretations for certain $\SL_2$-plethysm coefficients found 
in~\cite{Howe.1987, OSSZ.2024, GutCrystals}, among others.

\subsection{Recurrences for plethysm coefficients}
\label{ss:recursion}

In Section \ref{s:GeometricPicture}, we showed that the generating function $A_\mu(z,q)$ of $\SL_2$-plethysm coefficients is rational
and in Section \ref{section.GLn} we will show that
the generating function for plethysm coefficients $\aaa^\lambda_{\mu[\nu]}$
where $\nu$ is fixed is also rational.
These generating functions imply the existence of
certain linear recurrences, such as the ones found in \cite{OSSZ.2024,GutCrystals,Treteault}.
In this section we find a simple linear recurrence that fully describes $\SL_2$-plethysm coefficients.

Recall that
\begin{equation}
\label{equation.a in terms of quantum qbin}
	a_{w[h]}^{[k]} = \Qbinom{w+h}{h} (q-q^{-1})  \Big|_{q^k}
\end{equation}
for $1\leqslant k \leqslant hw+1$ and $k \equiv wh+1 \pmod 2$. We now state recursion relations for the plethysm coefficients~$a_{w[h]}^{[k]}$.

\begin{prop}
We have for $1\leqslant k \leqslant hw+1$ and $k \equiv wh+1 \pmod 2$
\begin{equation}
\label{equation.w recursion}
	a_{w[h]}^{[k]} = \begin{cases}
	a_{(w-1)[h]}^{[k-h]} & \text{if $wh+1-2w < k \leqslant wh+1$,}\\[.5em]
	a_{(w-1)[h]}^{[k-h]} + a_{w[h-1]}^{[k+w]} & \text{if $h < k\leqslant hw+1-2w$,}\\[.5em]
	a_{w[h-1]}^{[k+w]} - a_{(w-1)[h]}^{[h-k]} & \text{if $1 \leqslant k \leqslant h$.}
	\end{cases}
\end{equation}
Similarly, we have
\begin{equation}
\label{equation.h recursion}
	a_{w[h]}^{[k]} = \begin{cases}
	a_{w[h-1]}^{[k-w]} & \text{if $wh+1-2h < k \leqslant wh+1$,}\\[.5em]
	a_{w[h-1]}^{[k-w]} + a_{(w-1)[h]}^{[k+h]} & \text{if $w < k\leqslant hw+1-2h$,}\\[.5em]
	a_{(w-1)[h]}^{[k+h]} - a_{w[h-1]}^{[w-k]} & \text{if $1 \leqslant k \leqslant w$.}
	\end{cases}
\end{equation}
\end{prop}

\begin{proof}
The recursion~\eqref{equation.h recursion} follows from~\eqref{equation.w recursion} (and vice versa) by interchanging $h$ and $w$, since 
$a_{h[w]}^{[k]}$ is symmetric under interchanging $h$ and $w$.

The $q$-binomial coefficient satisfies the $q$-Pascal identity
\[
	\Qbinom{w+h}{h} = q^{h}\Qbinom{w+h-1}{h} + q^{-w} \Qbinom{w+h-1}{h-1}.
\]
Hence for $1\leqslant k \leqslant wh+1$ and $k \equiv wh+1 \pmod 2$ we have
\begin{multline*}
	a_{w[h]}^{[k]} = \left(q^h\Qbinom{w+h-1}{h}(q-q^{-1}) + q^{-w} \Qbinom{w+h-1}{h-1}(q-q^{-1})\right) \Big|_{q^k}\\
	= \Qbinom{w+h-1}{h}(q-q^{-1}) \Big|_{q^{k-h}} + \Qbinom{w+h-1}{h-1}(q-q^{-1}) \Big|_{q^{k+w}}.
\end{multline*}
When $wh+1-2w < k \leqslant wh+1$, the second term is zero and the first term is $a_{(w-1)[h]}^{[k-h]}$ by~\eqref{equation.a in terms of quantum qbin}, 
which is the first case in~\eqref{equation.w recursion}.
When $h < k\leqslant hw+1-2w$, the first term is $a_{(w-1)[h]}^{[k-h]}$ and the second term is $a_{w[h-1]}^{[k+w]}$, which
proves the second case in~\eqref{equation.w recursion}. 
When $1 \leqslant k \leqslant h$, the second term is still $a_{w[h-1]}^{[k+w]}$. For the first term we cannot use~\eqref{equation.a in terms of quantum qbin} 
to relate it to the plethysm coefficients. However, we can use the fact that
\[
	\Qbinom{w+h}{h} \Big|_{q^k} = \Qbinom{w+h}{h} \Big|_{q^{-k}}.
\]
Namely,
\[
\Qbinom{w+h-1}{h}(q-q^{-1}) \Big|_{q^{k-h}} = -\Qbinom{w-1+h}{h}  \Big|_{q^{h-k}}=- a_{(w-1)[h]}^{[h-k]}
\]
proving the last case in~\eqref{equation.w recursion}.
\end{proof}

\section{An explicit expression for \texorpdfstring{$A_\mu(z,q)$}{Aμ(z,q)}}
\label{section.Amu}

In Section~\ref{ss:GeneratingFunctions} we considered examples for rational functions for $A_\mu(z,q)$ and
saw that expressions with positive numerators lead to combinatorial interpretations counting points in cones.
In this section, we analyze the general form of $A_\mu(z,q)$ more closely.

\subsection{A conjecture for the denominator}

In Theorem~\ref{theorem.A QEhr}, we provided an explicit formula $A_\mu(z, q)$ and concluded in
Corollary~\ref{cor.A rational} that $A_\mu(z,q)$ is Elliott-rational. In this section, we study expressions for its
denominator.
Computing examples of $A_\mu(z,q)$ for $|\mu| \leqslant 10$, we
make the following conjecture about the form of the denominator
of $A_\mu(z,q)$.

\begin{conjecture} \label{conj.good denominator}
Fix a positive integer $w$ and a partition $\mu \vdash w$.
For any cell $c \in \mu$, recall the hook length of $c$ in the partition $\mu$ by $\mathsf{hl}_\mu(c)$.
If $w$ is even, define
$$d_\mu(z,q) = \prod_{c \in \mu} (1-z^{\mathsf{hl}_\mu(c)}) \prod_{i=0}^{\frac{w}{2}-1} (1-q^{w-2i}z)$$
and if $w$ is odd, define
$$d_\mu(z,q) = \prod_{c \in \mu} (1-z^{2\mathsf{hl}_\mu(c)}) \prod_{i=0}^{\frac{w-1}{2}} (1-q^{w-2i}z)~.$$
The expression $\frac{d_\mu(z,q) A_\mu(z,q)}{(1-z)^2}$ is an element of $\mathbb{Z}[z,q]$.
\end{conjecture}

The product of the hooks appearing in the denominator of $A_\mu(z,q)$ is suggestive,
however it is not exact and we are unable to provide a more precise conjecture to understand a link
between the product of hooks and the plethysm coefficients.

The expression $d_\mu(z,q)/(1-z)^2$ is not exactly the denominator
of $A_{\mu}(z,q)$, but it is very close and in certain cases we can state a more precise
expression for the denominator.

In the next section we provide an expression that the denominator
of $A_\mu(z,q)$ divides that can be determined using some
analysis of the expressions in Theorem~\ref{th:grsw_gf}.

\subsection{A theorem: the denominator divides the expression}

A direct calculation of the $\mathsf{PT}^q$ operator on $(q-q^{-1})\QEhr_\mu(z,q)$
as in the formula in Theorem \ref{theorem.A QEhr}
will require a partial fraction decomposition of $\QEhr_\mu(z,q)$
into terms for which we can apply the $\mathsf{PT}^q$.

We will show in the results below that $(q-q^{-1})\QEhr_\mu(z,q)$
can be expressed as a sum of terms of the form
\[
\frac{q^a z^b}{\mathsf{den}(z)(1-q^c z)}
\]
with
$\mathsf{den}(z) \in \mathbb{Q}\left[z\right]$,
$a,c \in \mathbb{Z}$ and $b \geqslant 0$.

If we wish to compute $\mathsf{PT}^q$ on
a single term in this partial fraction decomposition
we may factor out the $\frac{z^b}{\mathsf{den}(z)}$ and compute
the action on $\frac{q^a}{(1-q^c z)}$ which we show in Theorem \ref{th:semi-final}.

Let $\left< g_1, g_2 \right>_{qz}$ denote the ideal of $\mathbb{Q}[z,q,q^{-1}]$
generated by $g_1, g_2$.  In the following result we show that certain polynomials $p\in \mathbb{Q}[z,q,q^{-1}]$
are in the ideal $\left< g_1, g_2 \right>_{qz}$ by proving that there exists elements $f_1, f_2 \in \mathbb{Q}[z,q,q^{-1}]$ such that
$f_1 g_1 + f_2 g_2 = p$.

\begin{proposition} \label{prop:maincalc}
Let $a,b > 0$. Then
\[
1-z^{\frac{a+b}{\gcd(a,b)}} \in \left< 1-q^a z, 1- z/q^b\right>_{qz}.
\]
Furthermore if $a>b$, then
\[
1-z^{\frac{a-b}{\gcd(a,b)}} \in \left< 1-q^a z, 1-q^b z\right>_{qz}
\hbox{  and  }
1-z^{\frac{a-b}{\gcd(a,b)}} \in \left< 1-z/q^a, 1 - z/q^b\right>_{qz}~.
\]
\end{proposition}

\begin{proof}
We note that
\begin{align*}
\frac{\lcm(a,b)}{b}\pm\frac{\lcm(a,b)}{a} &= \frac{a \lcm(a,b) \pm b \lcm(a,b)}{ab}
= \frac{a\pm b}{\gcd(a,b)}~,
\end{align*}
where the last equality holds since $ab = \gcd(a,b) \lcm(a,b)$.

For the first identity, we choose
\[
f_1 = (1+q^a z+q^{2a} z^2+\cdots + q^{(\frac{\lcm(a,b)}{a}-1)a} z^{\frac{\lcm(a,b)}{a}-1})
\]
and
\[
f_2 = q^{\lcm(a,b)}z^{\frac{\lcm(a,b)}{a}}(1 + z/q^{b} + \cdots + z^{\frac{\lcm(a,b)}{b}-1}/q^{(\frac{\lcm(a,b)}{b}-1)b})~.
\]
We then calculate that
\begin{align*}
&f_1 (1-q^a z) + f_2 (1-z/q^b)\\
&= (1- q^{\lcm(a,b)} z^{\frac{\lcm(a,b)}{a}})
+ q^{\lcm(a,b)}z^{\frac{\lcm(a,b)}{a}}(1-z^{\frac{\lcm(a,b)}{b}}/q^{\lcm(a,b)})\\
&= 1-z^{\frac{\lcm(a,b)}{a}+\frac{\lcm(a,b)}{b}} = 1-z^{\frac{a+b}{\gcd(a,b)}}~.
\end{align*}

For the second identity, we choose
\[
f_1 = -z^{\frac{\lcm(a,b)}{b}-\frac{\lcm(a,b)}{a}}(1+q^a z + \cdots + q^{\lcm(a,b)-a} z^{\frac{\lcm(a,b)}{a}-1})
\]
and
\[
f_2 = (1+q^bz + \cdots + q^{\lcm(a,b)-b} z^{\frac{\lcm(a,b)}{b}-1})
\]
then calculate
\begin{align*}
&f_1 (1-q^a z) + f_2 (1-q^b z)\\
&= -z^{\frac{\lcm(a,b)}{b}-\frac{\lcm(a,b)}{a}} (1-q^{\lcm(a,b)} z^{\frac{\lcm(a,b)}{a}}) + (1-q^{\lcm(a,b)} z^{\frac{\lcm(a,b)}{b}})\\
&=1-z^{\frac{\lcm(a,b)}{b}-\frac{\lcm(a,b)}{a}}~.
\end{align*}
The third identity follows by replacing $q$ by $q^{-1}$.
\end{proof}

\begin{remark}\label{rem:div}
It also follows that if $1-z^c$ is in the ideal
and $c' = m c$, then
$1-z^{c'}$ is also in the ideal since
\[
1-z^{c'} = (1-z^{c})(1 + z^{c} + z^{2c} + \cdots + z^{(m-1)c})~.
\]
We will mostly apply this to the case that
$c = \frac{a \pm b}{\gcd(a,b)}$ and $c' = a \pm b$ or, in the case
when $a$ and $b$ are both even, $c' = \frac{a\pm b}{2}$.
\end{remark}

\begin{proposition}\label{prop:first_half}
For integers $r\in\mathbb{Z}$ and $\ell\geqslant 0$ such that
$r \neq \{0,-2,-4,\ldots, -2\ell\}$, let
\[
\mathsf{d}_{r,\ell} = (1-q^{r} z)(1-q^{r+2} z) \cdots (1-q^{r+2\ell} z).
\]
Then there exists non-zero polynomials
$p_{r,\ell,i} \in \mathbb{Z}[z,q,q^{-1}]$ for $0 \leqslant i \leqslant \ell$ such that
\begin{align*}
\sum_{i=0}^\ell p_{r,\ell,i} \frac{\mathsf{d}_{r,\ell}}{1-q^{r+2i}z} = (1-z^2)(1-z^4)\cdots(1-z^{2\ell})~.
\end{align*}
\end{proposition}

\begin{proof}
We prove the statement by induction on $\ell$.  For the base case $\ell=0$, we have $\frac{(1-q^{r} z)}{(1-q^{r} z)} = 1$
and hence the statement is true.

Now assume that the statement is true for any integer $r>0$ and for some fixed $\ell\geqslant 0$. We note that then
$\mathsf{d}_{r,\ell+1} = \mathsf{d}_{r,\ell}(1-q^{r+2\ell+2}z) = (1-q^{r}z)\mathsf{d}_{r+2,\ell}$ and hence
\begin{align}\label{eq:one}
\sum_{i=0}^\ell p_{r,\ell,i} \frac{\mathsf{d}_{r,\ell}(1-q^{r+2\ell+2}z)}{1-q^{r+2i}z}
= (1-z^2)(1-z^4)\cdots(1-z^{2\ell})(1-q^{r+2\ell+2}z)
\end{align}
and
\begin{align}\label{eq:two}
\sum_{i=0}^\ell p_{r+2,\ell,i} \frac{(1-q^{r}z)\mathsf{d}_{r+2,\ell}}{1-q^{r+2+2i}z}
= (1-z^2)(1-z^4)\cdots(1-z^{2\ell})(1-q^{r}z)~.
\end{align}

Now by Proposition \ref{prop:maincalc} and Note \ref{rem:div},
there exists $f_1,f_2 \in \mathbb{Q}[z,q,q^{-1}]$ such that
$f_1 (1-q^{r+2\ell+2}z) + f_2 (1-q^{r}z) = 1-z^{2\ell+2}$.  Multiply Equation~\eqref{eq:one} by $f_1$
and Equation~\eqref{eq:two} by $f_2$. Their sum is equal to
\begin{align} \label{eq:three}
\sum_{i=0}^{\ell+1} (f_1 p_{r,\ell,i} + f_2 p_{r+2,\ell,i-1}) \frac{\mathsf{d}_{r,\ell+1}}{1-q^{r+2i}z}
= (1-z^2)(1-z^4)\cdots(1-z^{2\ell+2})~.
\end{align}
This implies that we should choose $p_{r,\ell+1,i} = f_1 p_{r,\ell,i} + f_2 p_{r+2,\ell,i-1}$ (with the convention
that $p_{r+2,\ell,-1}=0$), however we need to be assured
that this $p_{r,\ell+1,i} \neq 0$.

To this end we note that the right hand side of Equation \eqref{eq:three} is a polynomial in $z$.  We also note
that $\frac{\mathsf{d}_{r,\ell+1}}{1-q^{r+2i}z}$ vanishes at all of the roots of the polynomial
in $q$ of $1-q^{r+2j}z$ for $j \neq i$.  Assume by way of contradiction that
$p_{r,\ell+1,i}=0$, then the left hand side is equal to zero for any $q$ equal to a root of $1-q^{r+2i}z$ while
the right hand side is non-zero. This proves the claim.
\end{proof}

In Proposition \ref{prop:first_half}, we can sharpen the right hand side if $r$ is even.

\begin{cor}\label{cor:sharp}
For integers $r>0$ such that $r$ is even and $\ell\geqslant 0$, let
$\mathsf{d}_{r,\ell} = (1-q^{r} z)(1-q^{r+2} z) \cdots (1-q^{r+2\ell} z)$. Then
there exists non-zero polynomials $p_{r,\ell,i} \in \mathbb{Z}[z,q]$ for $0 \leqslant i \leqslant \ell$ such that
\begin{align*}
\sum_{i=0}^\ell p_{r,\ell,i} \frac{\mathsf{d}_{r,\ell}}{1-q^{r+2i}z} = (1-z)(1-z^2)\cdots(1-z^{\ell}).
\end{align*}
\end{cor}

\begin{proof}
In the proof of Proposition \ref{prop:first_half}, since $r$ and $r+2\ell+2$ are
both even it is possible to choose $f_1$ and $f_2$ such
that $f_1(1-q^{r+2\ell+2}z) + f_2 (1-q^{r}z) = 1-z^{\ell+1}$ using Proposition \ref{prop:maincalc}
and Note \ref{rem:div}.
\end{proof}

\begin{proposition}\label{prop:evem_key}
For any $n,k\geqslant 0$ and $n+k > 0$ define
\[
\mathsf{d}'_{n,k} = (1-q^{2} z)(1-q^{4} z) \cdots (1-q^{2n} z) (1-z/q^2)(1-z/q^4)\cdots (1- z/q^{2k})~.
\]
There exists non-zero
polynomials $p_{n,k,i} \in \mathbb{Z}[z,q,q^{-1}]$ for $1 \leqslant i \leqslant n$
and $r_{n,k,i} \in \mathbb{Z}[z,q,q^{-1}]$ for $1 \leqslant i \leqslant k$ such that
\begin{align*}
\sum_{i=1}^{n}& p_{n,k,i} \frac{\mathsf{d}'_{n,k}}{1-q^{2i}z}
+\sum_{i=1}^k r_{n,k,i} \frac{\mathsf{d}'_{n,k}}{1-z/q^{2i}}= (1-z)(1-z^2)\cdots (1-z^{n+k})~.
\end{align*}
\end{proposition}

\begin{proof}
This result is proved by induction on $k$ and $n$. 

The base case $k=0$ is Corollary~\ref{cor:sharp} with $r=2$ and $\ell=n-1$.
However we need to multiply both sides of that result
by $1-z^{n}$ to agree with the base case.

The base case $n=0$ is again a special case of Corollary \ref{cor:sharp}
with $r=2$ and $\ell=k-1$
but with $q$ replaced by $q^{-1}$.
In that case we need to multiply both sides of the
equation by $1-z^{k}$ to agree with the base case.

Note that we have for positive $n$ and $k$,
\[
\mathsf{d}'_{n,k-1} (1- z/q^{2k}) = (1- q^{2n}z)\mathsf{d}'_{n-1,k} = \mathsf{d}'_{n,k}~.
\]

Now assume by way of induction that the proposition is true for smaller values of
either $k$ or $n$.
We have by the induction hypothesis that
\begin{align*}
&\sum_{i=1}^{n} p_{n,k-1,i} \frac{\mathsf{d}'_{n,k-1}(1- z/q^{2k})}{1-q^{2i}z}
+\sum_{i=1}^{k-1} r_{n,k-1,i} \frac{\mathsf{d}'_{n,k-1}(1- z/q^{2k})}{1-z/q^{2i}}\\
&= (1-z)(1-z^2)\cdots (1-z^{n+k})(1- z/q^{2k})~.\nonumber
\end{align*}
We also have by induction
\begin{align*}
&\sum_{i=1}^{n-1} p_{n-1,k,i} \frac{(1- q^{2n} z)\mathsf{d}'_{n-1,k}}{1-q^{2i}z}
+\sum_{i=1}^{k} r_{n-1,k,i} \frac{(1- q^{2n} z)\mathsf{d}'_{n-1,k}}{1-z/q^{2i}}\\
&= (1-z)(1-z^2)\cdots (1-z^{n+k})(1- q^{2n} z)~.\nonumber
\end{align*}

By Proposition \ref{prop:maincalc} and Note \ref{rem:div}
there exist $f_1, f_2 \in \mathbb{Q}[z,q,q^{-1}]$ such
that $f_1 (1- z/q^{2k}) + f_2 (1-q^{2n} z)= 1 - z^{n+k}$.
Therefore choosing $p_{n,k,i} = f_1 p_{n,k-1,i} + f_2 p_{n-1,k,i}$
for $1 \leqslant i \leqslant n$ (with the convention $p_{n-1,k,n} = 0$)
and $r_{n,k,i} = f_1 r_{n,k-1,i} + f_2 r_{n-1,k,i}$ for $1 \leqslant i \leqslant k$
(with the convention that $r_{n,k-1,k}=0$) satisfies the conditions
of the proposition for $n$ and $k$.  Hence, by induction the proposition holds.

The argument that all polynomials $p_{n,k,i}$ and $r_{n,k,i}$ are non-zero
is the same mutatis mutandi as the argument at the end of Proposition \ref{prop:first_half}.
\end{proof}

\begin{cor} \label{cor:even_case}
For $w>0$ and $w$ even, there exists polynomials
$s_{w,i} \in \mathbb{Q}[z,q,q^{-1}]$ for $0 \leqslant i \leqslant w+1$, $i \neq \frac{w}{2}$
such that
\[
\frac{1}{(1-q^w z)(1-q^{w-2}z) \cdots (1-q^{-w}z)}
= \sum_{\substack{i=0\\i\neq \frac{w}{2}}}^{w} \frac{s_{w,i}}{(1-q^{w-2i}z)\mathsf{den}_w(z)},
\]
where $\mathsf{den}_w(z) = (1-z)\cdot(1-z)(1-z^2)\cdots(1-z^w)$~.
\end{cor}

\begin{proof}
In Proposition \ref{prop:evem_key} choose $n=k=\frac{w}{2}$ and $s_{w,\frac{w}{2}-i} = p_{\frac{w}{2},\frac{w}{2}, i}$
for $1 \leqslant i < \frac{w}{2}$ and $s_{w,i} = r_{\frac{w}{2},\frac{w}{2}, i - \frac{w}{2}}$ for $\frac{w}{2} < i \leqslant w$.
Divide both sides of the equation by
\[
(1-q^w z)(1-q^{w-2}z) \cdots (1-q^{-w}z) \mathsf{den}_w(z)~.
\]
\end{proof}

The proof of the next proposition is similar to the last, but with minor modifications
because the exponents are not even.

\begin{cor}\label{cor:odd_case}
For $w \geqslant 0$ and $w$ odd, there exists polynomials $t_{w,i} \in
\mathbb{Q}[z,q,q^{-1}]$ for $0 \leqslant i \leqslant w$ such that
\[
\frac{1}{(1-q^w z)(1-q^{w-2}z) \cdots (1-q^{-w}z)}
= \sum_{i=0}^w \frac{t_{w,i}}{(1-q^{w-2i}z)\mathsf{den}_w(z)}
\]
where
$\mathsf{den}_w(z) = (1-z^2)(1-z^4)\cdots(1-z^{2w})$~.
\end{cor}

\begin{proof}  In Proposition \ref{prop:first_half}, we set
$r=-w$ and $\ell=w$.  The equation stated in the corollary is
precisely the equation in the proposition with $t_{w,i} = p_{-w,w,i}$
and both sides of the equation are divided by $\mathsf{d}_{-w,w} \mathsf{den}_w(z)$.
\end{proof}

We now state the main result of this section.

\begin{theorem}\label{th:semi-final}
For $w>0$ with $w$ even, define
\[
d_w(z,q) := (1-z)\prod_{i=1}^w (1-z^i) \prod_{i=1}^{\frac{w}{2}} (1-q^{2i} z)
\]
and for $w$ odd let
\[
d_w(z,q) := \prod_{i=1}^w (1-z^{2i}) \prod_{i=1}^{\frac{w+1}{2}} (1-q^{2i-1} z)~.
\]
Then for all $\mu \vdash w$,
\[
d_w(z,q) A_\mu(z,q) \in \mathbb{Z}[z,q]~.
\]
\end{theorem}

\begin{proof}
We use that $A_\mu(z,q) = \mathsf{PT}^q (q-q^{-1}) \QEhr_\mu(z,q)$.
By Theorem~\ref{th:grsw_gf}, Corollary \ref{cor:even_case} and Corollary \ref{cor:odd_case}, we have that
$\mathsf{den}_w(z) (q - q^{-1})\QEhr_\mu(z,q)$ is a sum of terms of the
form
\[
\frac{z^a q^b}{1-q^{w-2i} z}
\]
for $a \geqslant 0$, $b \in \mathbb{Z}$ and $0 \leqslant i \leqslant w$ and $i \neq \frac{w}{2}$.

Now we need to apply some analysis to compute the result of applying the $\mathsf{PT}^q$ operator
to these terms.  First we consider the terms such that $0 \leqslant i < \frac{w}{2}$. In this case
\[
\frac{z^a q^b}{1-q^{w-2i} z} = z^a q^b(1+q^{w-2i} z + q^{2(w-2i)} z^2 + q^{3(w-2i)} z^3 + \cdots )
\]
if $b\geqslant 0$. Hence $\mathsf{PT}^q \frac{z^a q^b}{1-q^{w-2i} z} = \frac{z^a q^b}{1-q^{w-2i} z}$.  Otherwise
$b$ is negative finite and then $\mathsf{PT}^q$ kills all the terms $z^a q^b q^{k(w-2i)} z^k$
such that $b + k(w-2i)<0$.  Let $d = \lfloor \frac{b}{w-2i} \rfloor$. Then
\begin{align}
\mathsf{PT}^q \frac{z^a q^b}{1-q^{w-2i} z}
&= \frac{z^a q^b}{1-q^{w-2i} z} - z^a q^b(1+q^{w-2i} z + q^{2(w-2i)} z^2 + \cdots
+ q^{d (w-2i)} z^d)\nonumber\\
&=\frac{z^a q^b}{1-q^{w-2i} z} - z^a q^b\frac{1-q^{(d+1)(w-2i)}z^{d+1}}{1-q^{w-2i} z}\nonumber\\
&=\frac{q^{(d+1)(w-2i)+b} z^{d+a+1}}{1-q^{w-2i} z}.
\label{eq:worst_case}
\end{align}
The important observation in that expression is that $(d+1)(w-2i)+b \geqslant 0$ and hence the
numerator is an element of $\mathbb{Z}[z,q]$.

Next we consider the terms with $\frac{w}{2} < i \leqslant w$.  If $b<0$, then
$\mathsf{PT}^q \frac{z^a q^b}{1-q^{w-2i} z} = 0$.  If $b \geqslant 0$, the operator
$\mathsf{PT}^q$ kills all the terms $z^a q^b q^{k(w-2i)} z^k$ with $b<k(w-2i)$. Let $d = \lfloor \frac{b}{w-2i} \rfloor$. Then
\[
\mathsf{PT}^q \frac{z^a q^b}{1-q^{w-2i} z} =
z^a q^b(1+q^{w-2i} z + q^{2(w-2i)} z^2 + \cdots + q^{d (w-2i)} z^d).
\]
We conclude that $\mathsf{PT}^q \mathsf{den}_w(z) (q-q^{-1}) \QEhr_\mu(z,q)$ is a sum of polynomials and
rational functions of the form seen in Equation \eqref{eq:worst_case} where the denominator is
of the form $1-q^{w-2i} z$ with $0 \leqslant i < \frac{w}{2}$.  Therefore if we multiply
this expression by $\prod_{i=1}^{\frac{w}{2}} (1-q^{2i} z)$ if $w$ is even or
$\prod_{i=1}^{\frac{w+1}{2}} (1-q^{2i-1} z)$ if $w$ is odd, then the resulting expression
is a polynomial.
\end{proof}

We have found limitations computing
the $\PT[q]$ operators on the expressions $(q-q^{-1}) \QEhr_\mu(z,q)$
as rational functions using computer packages that have implemented MacMahon operators.
The expressions can explode in complexity and
this prevented us from computing $A_\mu(z,q)$ beyond roughly $n=5$
using these implementations.
We were able to surpass these barriers using Theorem \ref{th:semi-final} and computing
the polynomial
\[
\sum_{h\leqslant m} (q-q^{-1})s_\mu\Big[\frac{q^h-q^{-h}}{q-q^{-1}}\Big] z^h
\]
for a sufficiently large $m$ and applying $\PT[q]$ to the polynomial.
Multiplying this expression by $d_\mu(z,q)$
produces a polynomial of the form $p_{\mathrm{low}}(z,q) + p_{\mathrm{high}}(z,q) \in \mathbb{Z}[q,z]$,
where the degrees of $p_{\mathrm{high}}(z,q)$ are larger than $m$ in $q$ and the
polynomial $p_{\mathrm{low}}(z,q)$ is independent of $m$.
Our computations suggest that taking $m \geqslant w^2$ is sufficient to compute
$p_{\mathrm{low}}(z,q)$, but this is by no means a sharp bound.
An expression for $A_\mu(z,q)$ is $\frac{p_{\mathrm{low}}(z,q)}{d_\mu(z,q)}$.

\section{Reciprocity}
\label{section.reciprocity}

In Section \ref{s:GeometricPicture}, we gave a geometric expression for $A_\mu(z,q)$ in terms of a quantum Ehrhart series and an integral operator 
$\PT[q](q-q^{-1})$ based on MacMahon's Partition Analysis.
As we saw in Section~\ref{ss:GeneratingFunctions}, the resulting rational function is not immediately a $q$-Ehrhart series, but for any given
 example we can express it as a positive sum of $q$-Ehrhart series of rational cones. We conjecture this is always possible.
\begin{conj}
\label{conjecture.polytope}
    The generating function $A_\mu(z,q)$ is a weighted generating function of lattice points of a disjoint union of convex rational half-open cones.
\end{conj}

In this section, we give supporting evidence for the conjecture. We begin by discussing a similar conjecture of Mulmuley.

\subsection{A conjecture of Mulmuley}
\label{sec:KM}
    In a collection of open problems and conjectures, Kirillov writes that one \textit{expects} the generating function
        \begin{equation}
            \label{eq:Kirillov}
        \mathsf{B}_{\mu[\nu]}^\lambda(t) = \sum_{n\geqslant0} \aaa_{n\mu[\nu]}^{n\lambda} t^n
        \end{equation}
    to be rational
    \cite[(\maltese) right after Conj.~2.14]{Kirillov}. This was proved by Mulmuley \cite[Thm.~1.6.1 (a)]{GCT6}, who furthermore conjectured the following.
    \begin{conjecture}[{\cite[Hypothesis~1.6.4 (1)]{GCT6}}]
        The generating function $\mathsf{B}_{\mu[\nu]}^\lambda(t)$ is the Ehrhart series of a polytope.
    \end{conjecture}
    Kahle and Michałek~\cite{KahleMichalek18} disproved the conjecture.
    \begin{thm}[{\cite{KahleMichalek18}}]
        Neither
        \[
        \mathsf{B}_{\mu[\nu]}^\lambda(t)
        = \sum_{n\geqslant0} \aaa_{n\mu[\nu]}^{n\lambda} t^n
        \quad\text{nor}\quad
        \mathsf{A}_{\mu[\nu]}^\lambda(t)
        = \sum_{n\geqslant0} \aaa_{\mu[n\nu]}^{n\lambda} t^n
        \]
        are the Ehrhart series of any union of closed convex polytopes.
    \end{thm}
    The counterexample is also a counterexample for the $\SL_2$ case of the conjecture:
    \begin{equation}
    \label{eq:KMcounterex}
    \mathsf{B}_{(4)[(3)]}^{(7,5)}(t)
    =
    \mathsf{A}_{(3)[(4)]}^{(7,5)}(t) 
    = \sum_{n\geqslant0} a_{3[4n]}^{[2n+1]} t^n.
    \end{equation}
    We outline their proof. Suppose there is a polytope $P$ for which \eqref{eq:KMcounterex} is the Ehrhart series. By Ehrhart reciprocity (cf.~Thm.~\ref{thm:Ehrhart reciprocity}),
    substituting $t$ for $t^{-1}$ in \eqref{eq:KMcounterex} gives the Ehrhart series of the interior $P^\circ$. After computing the coefficients of 
    both series explicitly, one obtains
    \[
    1 = \ehr_{P^\circ}(1) \leqslant \ehr_{P}(1) = 0,
    \]
    which is a contradiction. If the generating function was instead a positive sum of Ehrhart series, the contradiction would still hold.

    However, Kahle and Michałek do not consider the following two relaxed versions of the conjecture. The first is to half-open polytopes 
    (see Section \ref{ss:reciprocity}), which renders the Ehrhart reciprocity argument invalid. We remark that Mulmuley already considers 
    this weakening of the conjecture in \cite[\S3.1.2]{GCT6}. 
    
    The second is to consider series of the form 
    $\sum_{n\geqslant 0}\#\{v\in C\cap\ZZ^{w+1}\mid v_{w+1} = n\} t^n$ for some $(w+1)$-dimensional cone $C$. If the cone is pointed at 
    the origin, then this is the Ehrhart series of a $w$-dimensional polytope. Otherwise, it is a series of the form $t^k\Ehr_P(t)$ for some 
    polytope $P$, where $k$ is the $(w+1)$-st coordinate of the vertex of~$C$.

    Indeed, \cite[Thm.~5]{KahleMichalek18} implies that
    \[
    \mathsf{A}_{(3)[(4)]}^{(7,5)}(t) = \frac{t^{8} + t^{6} + 2 \, t^{4} + t^{2} + 1}{{\left(t^{6} - 1\right)} {\left(t^{3} - 1\right)}} = (t^{-1}+t^2) \Ehr_{(-\frac{1}{3},\frac{1}{3}]}(t)
    \]
    and its reciprocal is $\mathsf{A}_{(3)[(4)]}^{(7,5)}(t^{-1}) = (1+t^3)\Ehr_{[-\frac{1}{3},\frac{1}{3})}(t)$.
    Hence both of these are generating functions of lattice points is the union of two cones. We believe a weakening of Mulmuley's conjecture remains open.\medskip 

    We end by remarking that one can obtain their generating function from ours 
    by using the constant term operator 
    \[
	\CT[\lambda] \sum_{n=-\infty}^\infty a_n \lambda^n  = a_0
    \]
    as
    \[
    \mathsf{A}_{(3)[(4)]}^{(7,5)}(t) = \CT[q] \Big(\frac{1}{8q} \sum_{j,\ell} \ell A_3(j z,\ell q) {\Big|}_{z\mapsto t^{\frac{1}{4}}q^{-\frac{1}{2}}}\Big)\,,
    \]
    where $\ell$ ranges over $\{1,-1\}$ and $j$ ranges over the complex fourth roots of $1$, namely $j\in\{1,i,-1,-i\}$.
    The operations required to pass from one generating function to the other do not seem to have geometric interpretations. But note that the generating function $A_3(z,q)$ is shown to be the generating function of a union of two cones in Section \ref{ss:GeneratingFunctions}.

\subsection{The \texorpdfstring{$\hhh^*$}{h*} polynomial}

A major open problem in Ehrhart theory is to determine which rational functions are Ehrhart series of convex polytopes. A typical line of attack is to study the $\hhh^*$ polynomial of polytopes.
\begin{de}
Let $P$ be a rational polytope of dimension $w$.
The \emph{denominator of $P$} is the minimal integer $d$ such that $dP$ is a lattice polytope.
The Ehrhart series of $P$ is a rational function of the form
    \[
    \Ehr_P(z) = \frac{\hhh_{P}^*(z)}{(1-z^d)^{w+1}}
    \]
    for some polynomial $\hhh_{P}^*(z)$, which is called the \emph{$\hhh^*$ polynomial of $P$}.
\end{de}
When $d = 1$, the polytope is a lattice polytope. Hibi \cite{Hibi}, Stanley \cite{Stanley1991hilbert}, and Stapledon \cite{Stapledon} (among others) gave necessary conditions for a polynomial to be the $\hhh^*$ polynomial of a lattice polytope. Beck, Braun, and Vindas-Meléndez \cite{BBVM} generalized these to rational polytopes.
\begin{theorem}[$\aaa\bbb$-decomposition {\cite[Thm.~4.7]{BBVM}}]
    \label{thm: ab decomp}
    Let $P$ be a full-dimen\-sional rational polytope with denominator $d$ and let $\ell$ be the smallest positive integer such that $\ell P$ contains an interior lattice point. Then
    \[
    \frac{1+z+\cdots+z^{\ell-1}}{1+z+\cdots+z^{d-1}}
    \hhh^*_P(z) = \aaa(z) + z^\ell \bbb(z),
    \]
    where $\aaa(z)$ and $\bbb(z)$ are palindromic polynomials with nonnegative integer coefficients.
\end{theorem}
The following theorem is a consequence of the $\aaa\bbb$-decomposition.
\begin{theorem}[{\cite[Thm.~4.8]{BBVM}}]\label{thm:ineqs}
Let $P$ be a rational polytope of dimension $w$ and denominator $d$. Let $\hhh^*_P(z)$ be its $\hhh^*$ polynomial. Let $s = \deg \hhh^*_P(z)$ and $S = d(w+1)-1$. 
We have $s \leqslant S$. The coefficients $\hhh^*_0, \hhh^*_1, \ldots, \hhh^*_{S}$ of the $\hhh^*$ polynomial satisfy the following inequalities:
\begin{align*}
    \hhh^*_0 + \cdots + \hhh^*_{i+1} &\geqslant
    \hhh^*_S + \cdots + \hhh^*_{S-i} &&\text{for }~0\leqslant i < \lfloor S/2\rfloor,\\
    \hhh^*_s + \cdots + \hhh^*_{s-i} &\geqslant
    \hhh^*_0 + \cdots + \hhh^*_{i} &&\text{for }~0 \leqslant i \leqslant s.
\end{align*}
\end{theorem}
To our knowledge, the only generalization of the above results to $q$-Ehrhart series is \cite[Thm.~2.4]{DLVVW.2024}, which says that the coefficients (in $q$ and $z$) of the $\hhh^*$ polynomial are positive under very strict restrictions on the grading. This theorem does not seem to apply to our setting. See \cite{BajoEtAl} for a related recent direction.\medskip

The property of being a $q$-Ehrhart series is stronger than the property of Conjecture \ref{conj:Ehr}. Nevertheless, our tests reveal that 
$A_\mu(z,1)$ satisfies the restrictions imposed by Theorems \ref{thm: ab decomp} and~\ref{thm:ineqs}.
Given a partition $\mu$ of $w$, Theorem \ref{th:semi-final} implies that
\[
\hhh^*_\mu(z) := A_\mu(z,1)(1-z^d)^w
\]
is a polynomial, where $d$ is the lowest common multiple of the exponents of $z$ in $d_w(z,q)$.
If $\hhh^*_\mu(z)$ is the $\hhh^*$ polynomial of a rational polytope, then it would admit an $\aaa\bbb$-decomposition
\begin{equation}\label{eq:h*}
\frac{1+z+\cdots+z^{\ell-1}}{1+z+\cdots+z^{d-1}}
\hhh^*_\mu(z) = \aaa(z) + z^\ell \bbb(z)
\end{equation}
for palindromic polynomials with nonnegative coefficients $\aaa(z)$ and $\bbb(z)$, for some~$\ell$.

\begin{example}
Let $\mu=(3)$. Starting with the expression in Table~\ref{table.rational A} from Section~\ref{ss:GeneratingFunctions}, we obtain
\begin{align*}
A_3(z,1) 
&= \frac{z^2}{(1-z^4)(1-z)^2} + \frac{1}{(1-z^4)(1-z)} \\ 
&= \frac{z^8 + z^7 + 2z^6 + 3z^5 + 2z^4 + 3z^3 + 2z^2 + z + 1}{(1-z^4)^3}.
\end{align*}
Call this latter numerator $\hhh^*_3(z)$. We have
\[
\frac{1}{1+z+z^2+z^3}\hhh^*_3(z) = 
z^5 + z^3 + z^2 + 1.
\]
Note that the right-hand side is a palindromic polynomial. In particular, since the product of palindromic polynomials is palindromic, 
then $\hhh^*_\mu$ admits an $\aaa\bbb$-decomposition as in~\eqref{eq:h*}.
\end{example}

A computer check reveals that for all partitions $\mu$ with $|\mu| \leqslant 7$ which are either a hook or self-conjugate,
\[
\frac{1}{1+z+\cdots+z^{d-1}}\hhh^*_\mu(z)
\]
is a palindromic polynomial. In particular, this shows that they admit an $\aaa\bbb$-decomposition as in the above example.
For the remaining partitions of size $\leqslant 7$, we have checked that they satisfy the inequalities of Theorem~\ref{thm:ineqs}.

A possible reason why hooks and self-conjugate partitions might showcase a different behavior than arbitrary partitions is explored in the following section.

\subsection{Reciprocity theorems}
\label{ss:reciprocity}

Often in combinatorics, one is interested in a function or a polynomial $p(x)$ which has a combinatorial interpretation at each nonnegative integer value~$p(n)$. These include binomial coefficients, the hook-length formula, order polynomials of posets, chromatic polynomials, and many more.
Surprisingly, all of the above also have a combinatorial interpretation at each \emph{negative} integer. For instance, 
\[
\binom{n}{k} = \frac{n(n-1)\cdots(n-k+1)}{k(k-1)\cdots 1} 
\]
counts the number of $k$-subsets of $\{1,\ldots,n\}$, and
\[
\binom{-n}{k} = (-1)^k\frac{n(n+1)\cdots(n+k-1)}{k(k-1)\cdots 1} = (-1)^k\binom{n+k-1}{k} 
\]
counts the number of $k$-multisubsets of $\{1,\ldots,n\}$.
This is a \emph{combinatorial reciprocity theorem}. The term was introduced by Stanley in~\cite{Stanley1975reciprocity} and explored in depth by 
Beck and Sanyal in \cite{CRT}, where the authors show all above combinatorial reciprocity theorems (and more) as a consequence of Ehrhart theoretic equalities. The most basic of these equalities is as follows.
\begin{thm}[Ehrhart reciprocity]\label{thm:Ehrhart reciprocity}
    Let $P$ be a closed lattice polytope of dimension $w$. Then
    \[
    \Ehr_P(z^{-1}) = (-1)^{w+1}\Ehr_{P^\circ}(z),
    \]
    where $P^\circ$ is the interior of $P$.
\end{thm}
Chapoton \cite{Chapoton} gave a refinement for $q$-Ehrhart polynomials, which is then generalized in \cite{BK-Chapoton, DLVVW.2024}. We need a slight generalization to half-open polytopes. We state it with our notation (in particular, note that our weighting function is $\wt(v) = 2(v_1 + \cdots + v_w)$ and that $\wt(kv) = k\wt(v)$, and so it falls within~\cite[Def.~1.1]{DLVVW.2024}).

\begin{thm}[\cite{DLVVW.2024}, Thm. 2.9]
\label{thm:QEhr recip}
    Let $\bar P \subseteq \square$ be a closed rational convex polytope with faces $F_1, \ldots, F_k$. Let $P = \bar P \setminus\bigcup_1^j F_i$ and $P^\circ = \bar P \setminus\bigcup_{j+1}^k F_i$ be two complementary half-open polytopes. Then
    \[
    \QEhr_P(z^{-1},q^{-1}) = (-1)^{w+1}\QEhr_{P^\circ}(z,q).
    \]
\end{thm}

As a corollary, we obtain a reciprocity theorem for the quantum Ehrhart series indexed by the partition $\mu$. We provide a self-contained combinatorial proof.
\begin{prop}\label{prop:QEhr_reciprocity}
For all partitions $\mu \vdash w\in \mathbb{Z}_{>0}$, we have
    $$\QEhr_\mu(z^{-1},q^{-1}) = (-1)^{w+1} z^2 \QEhr_{\mu'}(z,q)~.$$
\end{prop}
\begin{proof}
The denominator for $\QEhr_\mu(z,q)$ from
Theorem \ref{th:grsw_gf}, with $q$ replaced by $q^{-1}$
and $z$ replaced by $z^{-1}$, is
\[
\prod_{i=0}^w(1-q^{-w+2i}/z)=
\prod_{i=0}^w(q^{w-2i}-1/z)=
(-z)^{-w-1}\prod_{i=0}^w(1-q^{w-2i}z)~.
\]
Recall from Section \ref{section.notation}, that we have $\des(T') = w-1-\des(T)$
and $\maj(T') = \binom{w}{2} - \maj(T)$.
Therefore, we compute directly
\begin{align*}
\QEhr_\mu(z^{-1},q^{-1})
&=
\frac{\sum_{T\in\SYT(\mu)} q^{2\maj(T)} (zq^w)^{-\des(T)}}{\prod_{i=0}^w(1-q^{-w+2i}/z)}\\
&=
(-z)^{w+1}\frac{\sum_{T\in\SYT(\mu)} q^{w(w-1) - 2\maj(T')} (zq^w)^{-(w-1)+\des(T')}}{\prod_{i=0}^w(1-q^{w-2i}z)}\\
&=
(-1)^{w+1}z^2\frac{\sum_{T\in\SYT(\mu)} q^{-2\maj(T')} (zq^w)^{\des(T')}}{\prod_{i=0}^w(1-q^{w-2i}z)}\\
&= (-1)^{w+1}z^2\QEhr_{\mu'}(z,q)~. \qedhere
\end{align*}
\end{proof}

We should remark that there is another relation
\[
\QEhr_\mu(z,q^{-1})
=\sum_{h \geqslant 0} s_\mu[ s_h](q^{-1},q) z^h
=\QEhr_\mu(z,q)~.
\]
This can be seen combinatorially with the Sch\"uztenberger involution
$\phi : \SYT(\mu) \rightarrow \SYT(\mu)$ since it is possible to show that
$\des(\phi(T)) = \des(T)$ and $|\phi(T)|\des(\phi(T))-2\maj(\phi(T)) = - |T|\des(T)+2\maj(T)$.

We now establish a combinatorial reciprocity theorem for $A_\mu(z,q)$,
as further evidence towards Conjecture~\ref{conjecture.polytope}.

\begin{thm}[Combinatorial reciprocity for $A_\mu$]
\label{thm:reciprocity}
For all partitions $\mu \vdash w\in \mathbb{Z}_{>0}$, we have
    $$A_\mu(z^{-1},q^{-1}) = (-1)^{w}z^{2} A_{\mu'}(z,q).$$
\end{thm}

The proof of Theorem~\ref{thm:reciprocity} 
was provided to us by Guoce Xin
and can be found after the following lemma.

\begin{lemma}
\label{lemma.partial fraction}
The partial fraction decomposition of $(q-q^{-1})\QEhr_\mu(z,q)$ has the following form:
\[
(q-q^{-1})\QEhr_\mu(z,q) = 
\sum_{j \geqslant 0} 
\left(\frac{g_j(z)}{1-q c_j z^{f_j}}
- \frac{c_j z^{f_j} g_j(z)}{q - c_j z^{f_j}}\right)
\]
with $g_j(z) \in \mathbb{C}(\!(z)\!)$, $c_j \in \mathbb{C}$,
$f_j \in \mathbb{Q}_{>0}$ with
$\sum_{j} g_j(z) = 0$.
\end{lemma}

\begin{proof}
The formula for $(q-q^{-1})\QEhr_\mu(z,q)$ is given in Theorem \ref{th:grsw_gf}.  We note that
we can factor the denominator $\prod_{i=0}^w (1-q^{w-2i}z)$
of this expression into pairs of linear factors of the form $(1 - c_j z^{f_j} q)(1 - c_j z^{f_j}/q)$
(except for potentially a single unpaired factor of $1-z$ if $w$ is even; this factor will be considered to
be part of the numerator $g_j(z)$).  The exponents $f_j$ in these expressions are all of the form ${\frac{1}{w-2k}}$
for $0 \leqslant k < \frac{w}{2}$. Similarly $c_j \in \mathbb{C}$ is a power of $e^{\frac{2\pi i}{2w-k}}$.

That is, we have
\[
(q-q^{-1})\QEhr_\mu(z,q) = \frac{L_\mu(z,q)}{\prod_j (1 - c_j z^{f_j} q)(1 - c_j z^{f_j}/q)}~,
\]
where 
\begin{equation}\label{eq:Lsum}
L_\mu(z,q) = (q-q^{-1}) (1-z)^{-\delta} \sum_{T\in\SYT(\mu)} q^{-2\maj(T)} (zq^w)^{\des(T)}
\end{equation}
and $\delta$ is equal to $1$ if $w$ is even and $\delta=0$ otherwise.
It is important to note here that $L_\mu(z,q^{-1}) = - L_\mu(z,q)$.  
Hence the degree in $q$ and the degree in $q^{-1}$ of $L_\mu(z,q)$ is the same.

The degree of $q$ in $L_\mu(z,q)$ in Equation \eqref{eq:Lsum} is the maximum value of
$1+\sum_{i\in\Des(T)} w - 2i$ as $T$ ranges over $\SYT(\mu)$.
Since $w-2i$ is positive for $i < w/2$ and negative otherwise, this sum is bounded above by
\[
1+\sum_{i=1}^{\lfloor w/2\rfloor} w - 2i =
1 + (w-2) + (w-4) + \cdots + \delta(w~\text{odd}).
\]
The degree of the denominator is
$w-1$ greater hence the rational function $(q-q^{-1})\QEhr_\mu(z,q)$
is proper in $q$ (that is, the degree in the numerator is strictly smaller
than the degree in the denominator for $w>1$).

Since the expression is proper in $q$, Heavyside's cover-up algorithm allows us to compute that
\begin{equation}
\label{eq:parfrac reciproc}
(q-q^{-1})\QEhr_\mu(z,q) = \sum_{j} \left(\frac{g_j(z)}{1 - c_j z^{f_j} q} + \frac{g'_j(z)}{1 - c_j z^{f_j}/q}\right)~,
\end{equation}
where we have that
\[
g_j(z) = \frac{L_\mu(z,\frac{1}{z^{f_j}c_j})}{(1-c_j^2 z^{2f_j})\prod_{j'\neq j} (1 - \frac{c_{j'}}{c_j} z^{f_{j'}-f_j})
(1 - c_{j'} c_j z^{f_{j'}+f_j})} = - g'_j(z).
\]

However, we have that the partial fraction decomposition is equal to
\begin{equation}
\label{equation.pt}
(q-q^{-1})\QEhr_\mu(z,q) = \sum_{j} \left(\frac{g_j(z)}{1 - c_j z^{f_j} q} - g_j(z) - \frac{c_j z^{f_j} g_j(z)}{q - c_j z^{f_j}}\right).
\end{equation}
Now since $(q-q^{-1})\QEhr_\mu(z,q)$ is proper we conclude that $- \sum_j g_j(z) = 0$.
\end{proof}

\begin{proof}[Proof of Theorem \ref{thm:reciprocity}]
We compute the $\PT[q]$ operator on a term in the expansion of
Lemma~\ref{lemma.partial fraction} using~\eqref{eq:parfrac reciproc} as
\begin{equation}\label{eq:Amu_term}
\PT[q] \left(\frac{g_j(z)}{1-q c_j z^{f_j}}
- \frac{g_j(z)}{1-c_j z^{f_j}/q}\right)
= \frac{g_j(z)}{1-q c_j z^{f_j}}
- g_j(z).
\end{equation}
We also compute $\PT[q]$ on the same expression
with $q$ replaced by $q^{-1}$ and $z$ replaced by $z^{-1}$ and obtain
\begin{align}\label{eq:QEhr_term}
\PT[q] &\left(\frac{g_j(1/z)}{1-c_j z^{-f_j}/q}
- \frac{g_j(1/z)}{1-q c_j z^{-f_j}}\right)
= \PT[q] \left(\frac{q z^{f_j} g_j(1/z)}{q z^{f_j} -c_j}
- \frac{z^{f_j} g_j(1/z)/q}{z^{f_j}/q- c_j}\right)\nonumber\\
&= \frac{q z^{f_j} g_j(1/z)}{q z^{f_j}-c_j}=\frac{g_j(1/z)}{1-c_j z^{-f_j}/q}~.
\end{align}
Next we compute $A_\mu(z^{-1},q^{-1})-\PT[q]( (q^{-1}-q)\QEhr_\mu(z^{-1},q^{-1}) )$.  To this end,
we replace $q$ by $q^{-1}$ and $z$ by $z^{-1}$ in Equation \eqref{eq:Amu_term} and subtract Equation \eqref{eq:QEhr_term}.
The result is $-g_j(1/z)$.

We therefore have
\[
A_\mu(z^{-1},q^{-1})-\PT[q]( (q^{-1}-q)\QEhr_\mu(z^{-1},q^{-1}) )
= -\sum_j g_j(1/z).
\]
This is equal to $0$ by Lemma \ref{lemma.partial fraction} and hence
by applying Proposition \ref{prop:QEhr_reciprocity}, we find
\begin{align*}
A_\mu(z^{-1},q^{-1}) &= \PT[q]( (q^{-1}-q)\QEhr_\mu(z^{-1},q^{-1}) )\\
&= -\PT[q]( (q-q^{-1})(-1)^{w+1} z^2 \QEhr_{\mu'}(z,q) )\\
&= (-1)^{w} z^2 A_{\mu'}(z,q)~.\qedhere
\end{align*}
\end{proof}

When $\mu$ is self-conjugate, the previous formula is a ``self-reciprocity'' statement of $A_\mu$. This phenomenon also arises for a few other families of partitions. We show this after a preparatory lemma.
\begin{lemma}\label{lem:hooks}
    Let $\mu^{0}$ be a self-conjugate partition, and let $n$ be the size of the Durfee square of $\mu^0$. Let $\mu = \mu^0 + (m^n)$ for some $m$. Then,
    \[
    s_{\mu}[s_h](q^{-1},q) = s_{\mu'}[s_{h + m}] (q^{-1},q).
    \]
\end{lemma}
\begin{proof}
    The result will follow from Stanley's hook-content formula \cite[Thm.~15.3]{Stanley_plane_partitions_2}
    \[
    s_\mu[s_h](q^{-1},q) = \prod_{\square\in \mu} \frac{[h+\mathsf{ct}(\square)]_q}{[\mathsf{hl}_\mu(\square)]_q},
    \]
    where $\mathsf{ct}(i,j) = j-i$ is the content and $\mathsf{hl}_\mu(i,j)$ is the hook length of the cell $(i,j)$ in the Young diagram of $\mu$. Note that $\prod_{\square\in\mu} [\mathsf{hl}_\mu(\square)]_q = \prod_{\square\in\mu'}[\mathsf{hl}_{\mu'}(\square)]_q$ for all $\mu$.
    The claim can therefore be proven by showing that for any partition $\mu$ constructed as in the statement we have
    \begin{equation}\label{eq:aux SHCF}
    \prod_{\square\in \mu} [h+\mathsf{ct}(\square)]_q
    =
    \prod_{\square\in \mu'} [h+m+\mathsf{ct}(\square)]_q.
    \end{equation}
    
    We expand the left-hand side as a product of $2n$ falling $q$-factorials, two for each cell in the main diagonal of~$\mu$. The cell $(i,i)$ corresponds on the one hand to the product of contents of all cells (weakly) to its right, giving rise to
    \[
    [h]_q\cdot[h+1]_q\cdot[h+2]_q\cdots[h+\mu_i-i]_q = \frac{[h+\mu_i-i]_q!}{[h-1]_q!},
    \]
    and on the other hand to the product of contents of all cells (strictly) below itself, giving rise to
    \[
    [h-1]_q\cdot[h-2]_q\cdot[h-3]_q\cdots[h-\mu'_i+i]_q = \frac{[h-1]_q!}{[h-\mu'_i+i-1]_q!}.
    \]
    Noting that $\mu_i = m + \mu^0_i$ and $\mu'_i = \mu^0_i$ for all $i$, we get that the left hand side of~\eqref{eq:aux SHCF} is
    \[
    \prod_{i=1}^n \frac{[h+\mu_i-i]_q!}{[h-\mu'_i+i-1]_q!} = 
    \prod_{i=1}^n \frac{[h+m+\mu^0_i-i]_q!}{[h-\mu^0_i+i-1]_q!}.
    \]
    The right-hand side of~\eqref{eq:aux SHCF} admits a similar expansion,
    \[
    \prod_{i=1}^n \frac{[h+m+\mu'_i-i]_q!}{[h+m-\mu_i+i-1]_q!} = 
    \prod_{i=1}^n \frac{[h+m+(\mu^0_i)-i]_q!}{[h+m-(\mu^0_i+m)+i-1]_q!},
    \]
    which is indeed equal to the expression for the left-hand side.
\end{proof}

\begin{remark}
    Any partition $\mu$ in one of the following families can either be constructed by following the procedure of Lemma~\ref{lem:hooks} 
    (for some choice of $\mu^0$ and $m$), or is the transpose of such a partition:
    \begin{enumerate}
        \item self-conjugate partitions,
        \item hook partitions,
        \item rectangular partitions, and
        \item partitions $\mu$ of empty $(|\mu|/2)$-core.
    \end{enumerate}
    This is a non-exhaustive list.
\end{remark}

\begin{cor}[Combinatorial self-reciprocity]
\label{cor:hooks}
If $\mu$ is a partition constructed as in Lemma~\ref{lem:hooks} or its transpose, then
    $$A_\mu(z^{-1},q^{-1}) = (-1)^{|\mu|}A_{\mu}(z,q)z^{\mu_1-\mu_1'+2}.$$
\end{cor}
\begin{proof}
Let $|\mu|=w$. We show the claim for $\mu$ of the form $\mu^0 + (m^n)$. Noting that $\mu_1 - \mu_1' = m$, we deduce that the claim also holds for $\mu'$. 
By the Lemma~\ref{lem:hooks}, for each $h$ we have
\[
s_{\mu}[s_h](q^{-1},q) = s_{\mu'}[s_{h + \mu_1 - \mu_1'}] (q^{-1},q),
\]
from which we obtain $a_{\mu[h]}^{[k]} = a_{\mu'[h + \mu_1 - \mu_1']}^{[k]}$ for all $h$.
We then have
\begin{align*}
    A_{\mu}(z^{-1},q^{-1}) 
    &\stackrel{\ref{thm:reciprocity}}{=} (-1)^w A_{\mu'}(z,q)z^2\\
    &= (-1)^w \sum_{h,k\geqslant0} a_{\mu'[h-2]}^{[k]} q^k z^h\\
    &= (-1)^w \sum_{h,k\geqslant 0} a_{\mu[h+\mu_1 - \mu_1'-2]}^{[k]}\ q^k z^h\\
    &= (-1)^w A_\mu(z,q) z^{\mu_1 - \mu_1'+2}. \qedhere
\end{align*}
\end{proof}

Computational evidence suggests that no other partitions satisfy this combinatorial self-reciprocity.

\section{Generating function of \texorpdfstring{$\GL_n$}{GLn}-plethysm coefficients}
\label{section.GLn}

In this section, we extend our approach to the multivariate generating function of 
$\GL_n$-plethysm coefficients, and define for a partition $\mu$,
\begin{equation}
\mathbb{A}_\mu(x_1,\ldots,x_n;y_1,\ldots,y_m)
= \sum_{\ell(\lambda)\leqslant n} \sum_{\ell(\nu)\leqslant m} \aaa_{\mu[\nu]}^\lambda x^\lambda y^\nu,
\end{equation}
and for partitions $\mu$ such that $\ell(\mu) \leqslant n$,
\begin{equation}
\mathbb{B}_\mu(x_1,\ldots,x_n;y_1,\ldots,y_m)
= \sum_{\ell(\lambda)\leqslant n} \sum_{\ell(\nu)\leqslant m} \aaa_{\nu[\mu]}^\lambda x^\lambda y^\nu,
\end{equation}
where $x^\lambda = x_1^{\lambda_1}x_2^{\lambda_2}\cdots x_n^{\lambda_{n}}$ and $y^\nu$ is similarly defined.

\begin{note}
The generating function $\mathbb{A}_\mu(x_1,\ldots,x_n;y_1\ldots,y_m)$ is similar to our $\SL_2$
plethysm generating function~\eqref{equation.Amu}.  These functions are related by the specialization
\[
A_\mu(z,q) = q \mathbb{A}_\mu(q,q^{-1};z)~.
\]
\end{note}

To compute this generalization of $A_\mu(z,q)$, we define an operator
on series in terms MacMahon's $\Omega_{\geqslant}$ operator
(discussed in Section \ref{ss:MacMahon}).
For a multivariate generating function $F(x_1, x_2, \ldots, x_n)$,
let $\alpha_1, \alpha_2, \ldots, \alpha_n$ be
a set of indeterminates and define
\[
\PT[x_1\geqslant x_2\geqslant \cdots\geqslant x_n] F(x_1, x_2, \ldots, x_n) = \Omega^{\alpha_1,\alpha_2, \ldots, \alpha_n}_\geqslant
F\left(x_1\alpha_1, \frac{x_2\alpha_2}{\alpha_1}, \ldots, \frac{x_n\alpha_n}{\alpha_{n-1}}\right).
\]
Notice that the effect of this operator when it acts on a single monomial
$x_1^{a_1} x_2^{a_2} \cdots x_n^{a_n}$ is that the result is $0$ unless
\[
a_1 \geqslant a_2 \geqslant \cdots \geqslant a_n \geqslant 0
\]
and it acts as the identity otherwise.  A useful property is the result by Xin \cite[Theorem 3.2]{Xin.2004}
that these operators preserve Elliot rationality when they act on
Elliot rational generating functions.

Another formula that we heavily rely on is the Cauchy identity (see \cite[Section I.4, equation (4.3)]{Macdonald}),
\[
\prod_{i = 1}^n \prod_{j=1}^m \frac{1}{1-x_i y_j} = \sum_{\lambda} s_\lambda[X_n]s_\lambda[Y_m]~.
\]

Now to construct our generating functions we introduce an operator $\mathcal{L}_{X_n}$
whose action on a symmetric function is
\[
\mathcal{L}_{X_n}(f[X_n]) = x^{-\rho^{(n)}} \PT[x_1\geqslant x_2\geqslant \cdots\geqslant x_n] (f[X_n] \Delta(X_n)),
\]
where $\rho^{(n)} = (n-1, n-2, \ldots, 1,0)$ and $\Delta(X_n) = \prod_{1 \leqslant i<j \leqslant n} (x_i - x_j)$.
This operator has the action on Schur functions indexed by a partition $\lambda$,
\[
\mathcal{L}_{X_n}( s_\lambda[X_n] ) = x^{-\rho^{(n)}} \PT[x_1\geqslant x_2\geqslant \cdots\geqslant x_n] \sum_{w \in S_n} (-1)^w
w( x^{\lambda + \rho^{(n)}}) = x^\lambda~.
\]
It also preserves Elliot rationality when it acts on generating functions of symmetric functions
because it is a composition of multiplication by the polynomial $\Delta(X_n)$, followed by an application of
$\PT[x_1\geqslant x_2\geqslant \cdots\geqslant x_n]$.

Define an application of a power sum $p_k$ plethysm applied
to the Cauchy identity to be the generating function
\begin{equation}
\label{equation:Pk expression}
\tilde{P}_{k}(x_1,\ldots,x_n;y_1,\ldots,y_m) =
\mathcal{L}_{Y_m} \Big( \prod_{i=1}^n \prod_{j=1}^m \frac{1}{1-x_i^k y_i} \Big)
= \sum_{\nu} s_\nu[p_k[X_n]] y^\nu = \sum_{\nu} p_k[s_\nu[X_n]] y^\nu~.
\end{equation}

We may compute an analogue indexed by partitions by setting
\[
\tilde{P}_{(k)}(x_1,\ldots,x_n;y_1,\ldots,y_m)
= \tilde{P}_{k}(x_1,\ldots,x_n;y_1,\ldots,y_m)
\]
and for partitions $\gamma$ of length greater than or equal to $2$, set
\begin{equation}\label{equation:P lambda expression}
\tilde{P}_{\gamma}(x_1,\ldots,x_n;y_1,\ldots,y_m) =
\CT[\alpha_1, \alpha_2, \ldots, \alpha_m] \Big(
\tilde{P}_{\gamma_1}(x_1,\ldots,x_n;\frac{y_1}{\alpha_1},\ldots,\frac{y_m}{\alpha_m})
\tilde{P}_{\overline{\gamma}}(x_1,\ldots,x_n;\alpha_1,\ldots,\alpha_m) \Big),
\end{equation}
where  $\alpha_1, \alpha_2, \ldots, \alpha_m$ is an extra set of variables
and $\overline{\gamma} = (\gamma_2, \gamma_3, \ldots, \gamma_{\ell(\gamma)})$.
We conclude by induction that for a partition $\gamma$
\begin{equation}\label{equation:P gamma expression}
\tilde{P}_{\gamma}(x_1,\ldots,x_n;y_1,\ldots,y_m) =
\sum_{\nu} p_\gamma[s_\nu[X_n]] y^\nu~.
\end{equation}

We are now ready to state our formula for
$\mathbb{A}_\mu(x_1,\ldots,x_n;y_1,\ldots,y_m)$.

\begin{thm}\label{thm:GLn one}
    For a partition $\mu$, the generating function
\[
    \mathbb{A}_\mu(x_1,\ldots,x_n;y_1,\ldots,y_m) =
    \mathcal{L}_{X_n}\Big(
    \sum_{\gamma} \frac{\chi^\mu(\gamma)}{z_\gamma}
    \tilde{P}_{\gamma}(x_1,\ldots,x_n;y_1,\ldots,y_m)
    \Big)~.
\]
\end{thm}

\begin{proof}
By Equation \eqref{equation:P gamma expression} the $\tilde{P}_{\gamma}(X_n;Y_m)$
are the generating functions for the plethysm $p_\gamma[s_\nu[X_n]]$.
Using the expansion of the Schur function in the power sum basis (see~\cite[Corollary 7.17.5]{EC2}),
$s_\mu[s_\nu[X_n]] = \sum_{\gamma} \frac{\chi^\mu(\gamma)}{z_\gamma} p_\gamma[s_\nu[X_n]]$,
we have that
\[
\sum_{\gamma} \frac{\chi^\mu(\gamma)}{z_\gamma}
    \tilde{P}_{\gamma}(x_1,\ldots,x_n;y_1,\ldots,y_m) = \sum_\nu s_\mu[s_\nu[X_n]] y^\nu~.
\]
It follows that
\begin{align*}
\mathcal{L}_{X_n}\Big(
    \sum_{\gamma} \frac{\chi^\mu(\gamma)}{z_\gamma}
    \tilde{P}_{\gamma}(x_1,\ldots,x_n;y_1,\ldots,y_m)
    \Big)
&= \sum_\nu \mathcal{L}_{X_n}(s_\mu[s_\nu[X_n]]) y^\nu\\
&= \sum_\nu \sum_{\lambda} \aaa^\lambda_{\mu[\nu]} \mathcal{L}_{X_n}(s_\lambda[X_n]) y^\nu\\
&= \sum_\nu \sum_{\lambda} \aaa^\lambda_{\mu[\nu]} x^\lambda y^\nu\\
&= \mathbb{A}_\mu(X_n;Y_m)~. \qedhere
\end{align*}
\end{proof}

\begin{example}
At least for small values of $n$ these generating functions
have some identifiable structure. We compute
\[
\mathcal{L}_{X_3}({\tilde P}_3(x_1,x_2,x_3;z)) =
\frac{1 - x_1^2 x_2 z}{(1-x_1^3 z)(1-x_1^3x_2^3 z^2)(1-x_1 x_2 x_3 z)},
\]
\[
\mathcal{L}_{X_3}({\tilde P}_{21}(x_1,x_2,x_3;z)) =
\frac{1+ x_1^6 x_2^3z^3}{(1-x_1^3 z)(1-x_1^4 x_2^2 z^2)(1+x_1^3x_2^3 z^2)(1+x_1 x_2 x_3z)},
\]
\[
\mathcal{L}_{X_3}({\tilde P}_{111}(x_1,x_2,x_3;z)) =
\frac{1- x_1^6 x_2^3z^3}{(1-x_1^3 z)(1-x_1^2 x_2 z)^2(1-x_1^3x_2^3 z^2)(1-x_1 x_2 x_3z)}~.
\]
The expressions for $\mathbb{A}_\lambda(x_1, x_2, x_3; z)$, with $|\lambda| = 3$ are
\[
\mathbb{A}_{(3)}(x_1,x_2,x_3;z) =
\frac{1+x_1^4 x_2^4 x_3 z^3}{p(x_1,x_2,x_3)}\left(1 + \frac{x_1^4 x_2^2 z^2}{1-x_1^2 x_2 z}\right),
\]
\[
\mathbb{A}_{(21)}(x_1,x_2,x_3;z) =
\frac{x_1^2 x_2 z}{(1-x_1^3 z)(1-x_1^2 x_2 z)(1-x_1 x_2 x_3 z)(1-x_1^3 x_2^3 z^2)},
\]
\[
\mathbb{A}_{(111)}(x_1,x_2,x_3;z) =
\frac{x_1x_2x_3z + x_1^3x_2^3z^2}{p(x_1,x_2,x_3)}\left(1 + \frac{x_1^4 x_2^2 z^2}{1-x_1^2 x_2 z}\right),
\]
where $p(x_1,x_2,x_3) = (1-x_1^3 z)(1-x_1^2 x_2^2 x_3^2 z^2)(1-x_1^6 x_2^6 z^4)$.
\end{example}

Our other generating function can be obtained without relying on the expressions $\tilde{P}_\gamma(X_n;Y_m)$.

\begin{thm}\label{thm:GLn two}
Let $\SSYT_n(\mu)$ denote the set of semi-standard Young tableaux of shape $\mu$
with entries in $\{1,2,\ldots, n\}$ and $x^T$ the weight of a $T \in\SSYT_n(\mu)$.
    The generating function
\[
    \mathbb{B}_\mu(x_1,\ldots,x_n;y_1,\ldots,y_m) = \mathcal{L}_{X_n} \mathcal{L}_{Y_m}\Big( \prod_{i=1}^m \prod_{T\in\SSYT_n(\mu)} \frac{1}{1-y_i x^T} \Big)~.
\]
\end{thm}

\begin{proof}
    Denote the monomial expansion
    of the Schur function indexed by the partition $\mu$ by the
    sum over semi-standard Young tableaux
    \[
    s_\mu[X_n] = \sum_{T\in\SSYT_n(\mu)} x^T
    \]
    of shape $\mu$.  The Cauchy identity is the rational function
    \[
    \prod_{i=1}^m \prod_{T\in\SSYT_n(\mu)} \frac{1}{1-y_i x^T} =
    \sum_{\nu} s_\nu[Y_m] s_\nu[s_\mu[X_n]],
    \]
    where the sum on the right hand side is over all partitions $\nu$ such that $\ell(\nu)\leqslant m$.
    We obtain $\mathbb{B}_\mu(x_1,\ldots,x_n;y_1,\ldots,y_m)$ by applying the
    $\mathcal{L}_{X_n}$ and $\mathcal{L}_{Y_m}$ operators which have the action
    $\mathcal{L}_{Y_m}( s_\nu[Y_m] ) = y^\nu$ and
\[
\mathcal{L}_{X_n}(s_\nu[s_\mu[X_n]]) = \sum_{\lambda} \aaa^\lambda_{\nu[\mu]} \mathcal{L}_{X_n}( s_\lambda[X_n])
 = \sum_{\lambda} \aaa^\lambda_{\nu[\mu]} x^\lambda~. \qedhere
\]
\end{proof}

\begin{cor}\label{cor:A and B are Elliot rational}
For any partition $\mu$ and positive integers $m,n$, the generating functions
$\mathbb{A}_\mu(X_n;Y_m)$ and $\mathbb{B}_\mu(X_n;Y_m)$ are Elliot rational.
\end{cor}

\begin{proof}
The Cauchy identity is a rational function in the variables
$x_1, x_2, \ldots, x_n$ and $y_1, y_2, \ldots, y_m$.
Since the operators $\CT[\alpha_1, \alpha_2, \ldots, \alpha_m]$
and $\PT[x_1\geqslant x_2\geqslant \cdots\geqslant x_n]$
preserve Elliot rationality by \cite[Theorem 3.2]{Xin.2004}, we
conclude that the operators $\mathcal{L}_{X_n}$ and $\mathcal{L}_{Y_m}$
also preserve Elliot rationality.  Therefore we have that
$\tilde{P}_\gamma(X_n; Y_m)$ are Elliot rational generating
functions as well as the expressions for $\mathbb{A}_\mu(X_n;Y_m)$ 
from Theorem \ref{thm:GLn one} and $\mathbb{B}_\mu(X_n;Y_m)$ from Theorem \ref{thm:GLn two}.
\end{proof}

Corollary \ref{cor:A and B are Elliot rational} at least partially answers
a comment made in \cite[\S3]{Treteault}
which states ``The author firmly believes that such recurrence formulas
can be found to compute $h_m[h_n]$ recursively for any $m$.''
since the existence of rational generating functions for $\GL_n$-plethysm
coefficients implies that they satisfy linear recurrences.
Tr\'eteault's recurrence for the $\SL_2$ plethysm coefficients $a^{[k]}_{3[h]}$
can be derived from the expression for $A_3(z,q)$ found in Table \ref{table.rational A}.

\begin{example}
In the paper \cite{BBDGK} the authors compute several examples of these series for $m=1$.  They note (in Example 2.1)
that the cases of $\mu = (2)$ and $(1,1)$ can easily be derived from identities in \cite[\S I.5, Example~5]{Macdonald}:
\[
\mathbb{B}_2(x_1, x_2, \ldots, x_n; y) = \prod_{i=1}^n \frac{1}{1-y^i (x_1 x_2 \cdots x_i)^2}
\]
and
\[
\mathbb{B}_{11}(x_1, x_2, \ldots, x_n; y) = \prod_{i=1}^{\lfloor n/2 \rfloor} \frac{1}{1-y^i (x_1 x_2 \cdots x_{2i})}~.
\]
In addition, they compute in Example 2.2
\[
\mathbb{B}_{3}(x_1, x_2; y) = \frac{1}{(1-y x_1^3)(1-y^4 x_1^6 x_2^6)} + \frac{x_1^4 x_2^2 y^2 }{(1-y x_1^3)(1-y x_1^2 x_2)(1-y^4 x_1^6 x_2^6)}~.
\]
Expressed in this way, the series can be used to obtain a combinatorial interpretation for the
plethysm coefficients $\aaa_{\mu[3]}^\lambda$ for $\ell(\lambda) \leqslant 2$ as the cardinality of a set of points.
\end{example}

These generating functions provide an approach for proving
the celebrated Foulkes conjecture, which states that
for integers $a,b$, if $a>b$ then,
\[
\left< s_a[s_b]-s_b[s_a], s_\lambda \right> \geqslant 0
\] for all partitions $\lambda$ of $ab$.
If we fix the integer $b$, then for all partitions $\lambda$ such that
$\ell(\lambda)>b$, then $\left< s_b[s_a], s_\lambda \right> =0$.
Therefore, to prove this conjecture for a fixed $b$ and all $a>b$, it is sufficient to show that
\[
\mathbb{B}_{(b)}(x_1,x_2,\ldots,x_b; y) - \mathbb{A}_{(b)}(x_1,x_2,\ldots,x_b;y)
= \sum_{a \geqslant 0} \sum_{\lambda : \ell(\lambda)\leqslant b} \left< s_a[s_b]-s_b[s_a], s_\lambda \right> x^\lambda y^a
\]
has positive coefficients for $a>b$ (note that for $a<b$, the coefficients will be negative). 
The conjecture is trivial for $b\leqslant 2$ and known to hold for $b\leqslant 5$ \cite{CIM} (as well as for some other regimes of parameters).
In Appendix~\ref{appendix.B} we state the results of a computer calculation
for this difference and a positive expression for $b=3$,
thereby proving Foulkes conjecture for $b=3$ and all $a>3$.
\medskip

As another corollary, we obtain a shorter proof of the rationality of Kirillov's generating function \eqref{eq:Kirillov}, first established by 
Mulmuley in \cite[Thm.~1.6.1~(a)]{GCT6}.

\begin{cor}
For fixed partitions, $\lambda$, $\mu$, $\nu$ such that $|\lambda| = |\mu| \cdot |\nu|$,
the generating function
\[
{\mathsf{B}}^\lambda_{\mu[\nu]}(t) = \sum_{r \geqslant 0} \aaa_{r\mu[\nu]}^{r \lambda} t^r
\]
is a rational function in $t$,
where we have used the notation $r \lambda = (r \lambda_1, r \lambda_2, \ldots, r \lambda_\ell)$.
\end{cor}

\begin{proof}
Apply the operator $\CT[x_1,x_2, \ldots, x_n]$ to the expression
\[
\frac{\mathbb{B}_\nu(x_1,\ldots,x_n;y_1,\ldots,y_n)}{(1-t x^{-\lambda})(1-y^{-\mu})}
= \sum_{\gamma : \ell(\gamma) \leqslant n}
\sum_{\tau : \ell(\tau) \leqslant n}
\sum_{r \geqslant 0}
\sum_{d \geqslant 0}
\aaa_{\gamma[\nu]}^{\tau} t^r x^{\tau-r\lambda} y^{\gamma-d\mu}.
\] 
The result is an Elliott rational function by \cite[Theorem 3.2]{Xin.2004}
and the $\CT[x_1,x_2, \ldots, x_n]$ operator kills all terms except when $\tau = r \lambda$ for some $r$.

Next apply the operator $\CT[y_1,y_2, \ldots, y_n]$ and the only terms
that remain after the application of this operator are those that
satisfy $\gamma = d \mu$ and $|r \lambda| = |d \mu| \cdot |\nu|$. Again, the
resulting expression is Elliott rational by \cite[Theorem 3.2]{Xin.2004}
and equal to ${\mathsf{B}}^\lambda_{\mu[\nu]}(t)$~.
\end{proof}

\appendix
\section{Expressions for \texorpdfstring{$A_\mu(z,q)$}{Aμ(z,q)} for \texorpdfstring{$|\mu|=5$}{|μ|=5}}
\label{appendix.A}

An expression for $A_\mu(z,q)$ with a common denominator will generally not be positive.
In the formula below we have provided an expansion of $A_\mu(z,q)$
for $\mu \in \{ (5), (4,1), (3,2), (3,1,1)\}$.
The others for a partition of $5$
can be obtained by applying Theorem \ref{thm:reciprocity}.

\begin{align*}
A_5&(z,q) = \frac{ q^{7} z^{4} + q^{7} z^{2} + q^{6} z^{3} }{ {\left(1-z^{8}\right)} {\left(1-q^{5} z\right)} {\left(1-q^{3} z\right)} }
+
\frac{ q^{5} z^{4} }{ {\left(1-z^{8}\right)} {\left(1-q^{5} z\right)} {\left(1-q^{3} z\right)} {\left(1-q z\right)} }\\
&+
\frac{ q^{5} z^{8} }{ {\left(1-z^{8}\right)} {\left(1-q^{5} z\right)} {\left(1-q^{3} z\right)} {\left(1-q z\right)} {\left(1-z^{2}\right)} }
+ 
\frac{ q^{3} z^{8} }{ {\left(1-z^{8}\right)} {\left(1-q^{5} z\right)} {\left(1-z^{4}\right)} }\\
&+
\frac{ q^{7} z^{8} + q^{6} z^{7} + q^{5} z^{6} }{ {\left(1-z^{8}\right)} {\left(1-q^{5} z\right)} {\left(1-q^{3} z\right)} {\left(1-z^{2}\right)} }
+
\frac{ q^3z^{12} + q^3z^2 }{ {\left(1-z^8\right)} {\left(1-q^{5} z\right)} {\left(1-z^4\right)}^{2} }\\
&+
\frac{ q^{7} z^{6} + q^{6} z^{7} + q^{5} z^{4} }{ {\left(1-z^{8}\right)} {\left(1-q^{5} z\right)} {\left(1-q^{3} z\right)} {\left(1-z^{4}\right)} }
+
\frac{ q^{4} z^{5} }{ {\left(1-z^{8}\right)} {\left(1-q^{5} z\right)} {\left(1-z^{4}\right)} {\left(1-z^{2}\right)} }\\
&+
\frac{ q^{7} z^{6} + q^{6} z^{7} + q^{5} z^{8} }{ {\left(1-z^{8}\right)} {\left(1-q^{5} z\right)} {\left(1-q^{3} z\right)} {\left(1-z^{4}\right)} {\left(1-z^{2}\right)} }
+
\frac{ q^{4} z^{9} + q^{4} z^{3} }{ {\left(1-z^8\right)} {\left(1-q^{5} z\right)} {\left(1-z^{4}\right)} }\\
&+
\frac{ q^{5} z^{22} + q^{5} z^{20} + q^{5} z^{16} + q^{5} z^{14} }{ {\left(1-z^{12}\right)}{\left(1-z^{8}\right)} {\left(1-q^{5} z\right)} {\left(1-z^{4}\right)} {\left(1-q z\right)} }
+
\frac{ q^{2} z^{5} }{ {\left(1-z^{8}\right)} {\left(1-q^{5} z\right)} {\left(1-z^{2}\right)} }\\
&+ 
\frac{ qz^{18} + q }{ {\left(1-z^{12}\right)} {\left(1-z^{8}\right)} {\left(1-q^{5} z\right)} {\left(1-z^{4}\right)} }
+
\frac{ q^{2} z^{11} }{ {\left(1-z^{8}\right)} {\left(1-q^{5} z\right)} {\left(1-z^{6}\right)} {\left(1-z^{2}\right)} }\\
&+ 
\frac{ q^{4} z^{15} }{ {\left(1-z^{12}\right)} {\left(1-z^{8}\right)} {\left(1-q^{5} z\right)} }
+
\frac{ q^{4} z^{19} }{ {\left(1-z^{12}\right)} {\left(1-z^{8}\right)} {\left(1-q^{5} z\right)} {\left(1-z^{2}\right)} }
\end{align*}

\begin{align*}
A_{41}(z,q)=&\frac{q^5 z^6 + q^4 z + q^2 z^3}{(1-q^5z)(1-q^3z)(1-qz)(1-z^6)(1-z^2)}
+\frac{q^3 z^2}{(1-q^5z)(1-q^3z)(1-qz)(1-z^4)(1-z^2)}\\
&+\frac{q^3 z^8}{(1-q^5z)(1-qz)(1-z^6)(1-z^4)(1-z^2)}
+\frac{q^3 z^6 + q^2 z^7}{(1-q^3z)(1-qz)(1-z^6)(1-z^4)(1-z^2)}\\
&+\frac{q z^4+q^2 z^5}{(1-q^5z)(1-q^3z)(1-z^6)(1-z^4)(1-z^2)}
\end{align*}

\begin{align*}
&A_{32}(z,q)=\\
&\frac{q^5z^6 + q z^{6}+ q^{3} z^{10}+ q^{4} z^{5}+ q^{4} z^{3}+ q^3 z^{12}}{(1-q^5z)(1-qz)(1-z^8)(1-z^6)(1-z^4)}
+ \frac{q^{6} z^{3} + q^6 z^{13}+ 2q^8z^7+ q^7 z^8}{(1-q^5z)(1-q^3z)(1-qz)(1-z^8)(1-z^6)}\\
&+\frac{q^{5} z^{6} + q^7z^6 + q^7z^{10}+ q^{5} z^{8}}{(1-q^5z)(1-q^3z)(1-qz)(1-z^8)(1-z^4)}
+\frac{q z^{4} + q^{2} z^{7}+ q^{2} z
+q^{3} z^{6} + q z^{12}+ 2 q^{3} z^{6} + 2 q z^{8}+ q^{3} z^{8}}{(1-q^3z)(1-qz)(1-z^8)(1-z^6)(1-z^4)}\\  
&+\frac{q^7z^6 + q^8z^7  + q^4z^5+ q^{5} z^{4} + q^7z^2+ q^4 z^7}{(1-q^5z)(1-q^3z)(1-qz)(1-z^6)(1-z^4)}
+\frac{q^{2} z^{3} + q^3 z^4 + q^4z^5+ q^6z^5  + q^{3} z^{4}
+q^{2} z^{11}+ q^{2} z^{11}}{(1-q^5z)(1-q^3z)(1-z^8)(1-z^6)(1-z^4)}
\end{align*}

\begin{align*}
A_{311}(z,q)=&\frac{q^{3} z^{4} + q^{2} z^{5} + q z^{6}}{(1-q^3z)(1-qz)(1-z^4)^2(1-z^2)}
+\frac{q^{4} z^{5} + q^{5} z^{2}}{(1-q^5z)(1-q^3z)(1-qz)(1-z^4)(1-z^2)}\\
&+\frac{q^{8} z^{7} + q z^{2} + q^{3} z^{4} + q^{2} z^{5}+ q^{6} z^{5}+ q^{7} z^{4}}{(1-q^5z)(1-q^3z)(1-qz)(1-z^4)^2}
\end{align*}

\section{Foulkes Conjecture for $b=3$}\label{appendix.B}

The generating function of the Foulkes coefficients $\langle s_a[s_3] - s_3[s_a], s_\lambda\rangle$ can be written as a sum of positive rational functions
\[
\mathbb{B}_{3}(x_1,x_2,x_3;z) - \mathbb{A}_3(x_1,x_2,x_3;z) + x_1^2 x_2^2 x_3^2 z^2
= \sum_{(m,S) \in L} \frac{m}{\prod_{t \in S} (1-t)}
\]
where $L$ is the following list of 86 pairs:
\begin{longtable}{ll|ll}
$m$ & $S$ & $m$ & $S$ \\
\hline
$x_1^{23}x_2^{15}x_3^{13}z^{17}$ & $\{m_1,m_2,m_4,m_5,m_6,m_7,m_8\}$ & $x_1^{17}x_2^{10}x_3^{6}z^{11}$ & $\{m_1,m_2,m_4,m_5,m_6,m_7,m_8\}$ \\
$x_1^{26}x_2^{19}x_3^{12}z^{19}$ & $\{m_1,m_2,m_3,m_5,m_6,m_7,m_8\}$ & $x_1^{20}x_2^{12}x_3^{10}z^{14}$ & $\{m_1,m_2,m_3,m_4,m_6,m_7,m_8\}$ \\
$x_1^{21}x_2^{15}x_3^{12}z^{16}$ & $\{m_1,m_2,m_3,m_4,m_5,m_6,m_8\}$ & $x_1^{22}x_2^{13}x_3^{7}z^{14}$ & $\{m_0,m_2,m_3,m_4,m_6,m_7,m_8\}$ \\
$x_1^{31}x_2^{27}x_3^{14}z^{24}$ & $\{m_0,m_1,m_3,m_5,m_6,m_7,m_8\}$ & $x_1^{30}x_2^{23}x_3^{13}z^{22}$ & $\{m_0,m_1,m_3,m_4,m_6,m_7,m_8\}$ \\
$x_1^{23}x_2^{20}x_3^{8}z^{17}$ & $\{m_0,m_1,m_3,m_4,m_5,m_6,m_8\}$ & $x_1^{29}x_2^{22}x_3^{12}z^{21}$ & $\{m_0,m_1,m_2,m_4,m_5,m_7,m_8\}$ \\
$x_1^{26}x_2^{21}x_3^{13}z^{20}$ & $\{m_0,m_1,m_2,m_4,m_5,m_7,m_8\}$ & $x_1^{20}x_2^{19}x_3^{12}z^{17}$ & $\{m_0,m_1,m_2,m_3,m_5,m_6,m_8\}$ \\
$x_1^{18}x_2^{17}x_3^{7}z^{14}$ & $\{m_0,m_1,m_2,m_3,m_5,m_6,m_7\}$ & $x_1^{11}x_2^{9}x_3^{4}z^{8}$ & $\{m_0,m_1,m_2,m_3,m_5,m_6,m_7\}$ \\
$x_1^{9}x_2^{8}x_3^{4}z^{7}$ & $\{m_0,m_1,m_2,m_3,m_5,m_6,m_7\}$ & $x_1^{8}x_2^{6}x_3^{4}z^{6}$ & $\{m_0,m_1,m_2,m_3,m_5,m_6,m_7\}$ \\
$x_1^{21}x_2^{18}x_3^{15}z^{18}$ & $\{m_0,m_1,m_2,m_3,m_4,m_7,m_8\}$ & $x_1^{24}x_2^{21}x_3^{12}z^{19}$ & $\{m_0,m_1,m_2,m_3,m_4,m_6,m_8\}$ \\
$x_1^{20}x_2^{19}x_3^{9}z^{16}$ & $\{m_0,m_1,m_2,m_3,m_4,m_6,m_8\}$ & $x_1^{24}x_2^{21}x_3^{9}z^{18}$ & $\{m_0,m_1,m_2,m_3,m_4,m_5,m_8\}$ \\
$x_1^{13}x_2^{12}x_3^{5}z^{10}$ & $\{m_0,m_1,m_2,m_3,m_4,m_5,m_8\}$ & $x_1^{31}x_2^{25}x_3^{19}z^{25}$ & $\{m_0,m_1,m_2,m_3,m_4,m_5,m_7\}$ \\
$x_1^{31}x_2^{25}x_3^{16}z^{24}$ & $\{m_0,m_1,m_2,m_3,m_4,m_5,m_7\}$ & $x_1^{22}x_2^{19}x_3^{13}z^{18}$ & $\{m_0,m_1,m_2,m_3,m_4,m_5,m_7\}$ \\
$x_1^{14}x_2^{10}x_3^{9}z^{11}$ & $\{m_0,m_1,m_2,m_3,m_4,m_5,m_6\}$ & $x_1^{10}x_2^{8}x_3^{6}z^{8}$ & $\{m_0,m_1,m_2,m_3,m_4,m_5,m_6\}$ \\
$x_1^{11}x_2^{8}x_3^{5}z^{8}$ & $\{m_3,m_4,m_5,m_6,m_7,m_8\}$ & $x_1^{11}x_2^{8}x_3^{8}z^{9}$ & $\{m_2,m_3,m_5,m_6,m_7,m_8\}$ \\
$x_1^{18}x_2^{8}x_3^{7}z^{11}$ & $\{m_2,m_3,m_4,m_6,m_7,m_8\}$ & $x_1^{19}x_2^{12}x_3^{8}z^{13}$ & $\{m_2,m_3,m_4,m_5,m_7,m_8\}$ \\
$x_1^{8}x_2^{4}x_3^{3}z^{5}$ & $\{m_2,m_3,m_4,m_5,m_7,m_8\}$ & $x_1^{15}x_2^{11}x_3^{7}z^{11}$ & $\{m_2,m_3,m_4,m_5,m_6,m_7\}$ \\
$x_1^{11}x_2^{6}x_3^{4}z^{7}$ & $\{m_1,m_3,m_5,m_6,m_7,m_8\}$ & $x_1^{19}x_2^{15}x_3^{11}z^{15}$ & $\{m_1,m_3,m_4,m_5,m_6,m_8\}$ \\
$x_1^{24}x_2^{21}x_3^{15}z^{20}$ & $\{m_1,m_3,m_4,m_5,m_6,m_7\}$ & $x_1^{13}x_2^{8}x_3^{6}z^{9}$ & $\{m_1,m_3,m_4,m_5,m_6,m_7\}$ \\
$x_1^{12}x_2^{10}x_3^{8}z^{10}$ & $\{m_1,m_2,m_5,m_6,m_7,m_8\}$ & $x_1^{10}x_2^{6}x_3^{5}z^{7}$ & $\{m_1,m_2,m_4,m_6,m_7,m_8\}$ \\
$x_1^{19}x_2^{15}x_3^{11}z^{15}$ & $\{m_1,m_2,m_3,m_6,m_7,m_8\}$ & $x_1^{16}x_2^{9}x_3^{8}z^{11}$ & $\{m_1,m_2,m_3,m_6,m_7,m_8\}$ \\
$x_1^{19}x_2^{15}x_3^{14}z^{16}$ & $\{m_1,m_2,m_3,m_5,m_6,m_8\}$ & $x_1^{19}x_2^{15}x_3^{8}z^{14}$ & $\{m_1,m_2,m_3,m_5,m_6,m_8\}$ \\
$x_1^{14}x_2^{9}x_3^{7}z^{10}$ & $\{m_1,m_2,m_3,m_5,m_6,m_7\}$ & $x_1^{23}x_2^{17}x_3^{14}z^{18}$ & $\{m_1,m_2,m_3,m_4,m_7,m_8\}$ \\
$x_1^{23}x_2^{17}x_3^{14}z^{18}$ & $\{m_1,m_2,m_3,m_4,m_6,m_8\}$ & $x_1^{20}x_2^{16}x_3^{12}z^{16}$ & $\{m_1,m_2,m_3,m_4,m_6,m_8\}$ \\
$x_1^{17}x_2^{10}x_3^{9}z^{12}$ & $\{m_1,m_2,m_3,m_4,m_6,m_7\}$ & $x_1^{12}x_2^{6}x_3^{6}z^{8}$ & $\{m_1,m_2,m_3,m_4,m_6,m_7\}$ \\
$x_1^{10}x_2^{10}x_3^{7}z^{9}$ & $\{m_1,m_2,m_3,m_4,m_5,m_7\}$ & $x_1^{10}x_2^{4}x_3^{4}z^{6}$ & $\{m_0,m_3,m_5,m_6,m_7,m_8\}$ \\
$x_1^{9}x_2^{6}x_3^{3}z^{6}$ & $\{m_0,m_3,m_4,m_5,m_7,m_8\}$ & $x_1^{21}x_2^{15}x_3^{12}z^{16}$ & $\{m_0,m_2,m_4,m_5,m_6,m_7\}$ \\
$x_1^{10}x_2^{8}x_3^{3}z^{7}$ & $\{m_0,m_2,m_3,m_5,m_7,m_8\}$ & $x_1^{9}x_2^{6}x_3^{6}z^{7}$ & $\{m_0,m_2,m_3,m_5,m_6,m_7\}$ \\
$x_1^{12}x_2^{12}x_3^{9}z^{11}$ & $\{m_0,m_2,m_3,m_4,m_7,m_8\}$ & $x_1^{7}x_2^{3}x_3^{2}z^{4}$ & $\{m_0,m_2,m_3,m_4,m_7,m_8\}$ \\
$x_1^{22}x_2^{19}x_3^{10}z^{17}$ & $\{m_0,m_2,m_3,m_4,m_5,m_7\}$ & $x_1^{12}x_2^{12}x_3^{6}z^{10}$ & $\{m_0,m_2,m_3,m_4,m_5,m_6\}$ \\
$x_1^{23}x_2^{20}x_3^{11}z^{18}$ & $\{m_0,m_1,m_5,m_6,m_7,m_8\}$ & $x_1^{22}x_2^{16}x_3^{10}z^{16}$ & $\{m_0,m_1,m_4,m_6,m_7,m_8\}$ \\
$x_1^{14}x_2^{11}x_3^{8}z^{11}$ & $\{m_0,m_1,m_4,m_6,m_7,m_8\}$ & $x_1^{12}x_2^{12}x_3^{6}z^{10}$ & $\{m_0,m_1,m_4,m_5,m_7,m_8\}$ \\
$x_1^{12}x_2^{12}x_3^{3}z^{9}$ & $\{m_0,m_1,m_4,m_5,m_6,m_8\}$ & $x_1^{10}x_2^{7}x_3^{4}z^{7}$ & $\{m_0,m_1,m_3,m_6,m_7,m_8\}$ \\
$x_1^{20}x_2^{19}x_3^{9}z^{16}$ & $\{m_0,m_1,m_3,m_5,m_7,m_8\}$ & $x_1^{10}x_2^{10}x_3^{4}z^{8}$ & $\{m_0,m_1,m_3,m_5,m_7,m_8\}$ \\
$x_1^{12}x_2^{9}x_3^{6}z^{9}$ & $\{m_0,m_1,m_3,m_5,m_6,m_7\}$ & $x_1^{14}x_2^{11}x_3^{5}z^{10}$ & $\{m_0,m_1,m_3,m_4,m_7,m_8\}$ \\
$x_1^{14}x_2^{14}x_3^{8}z^{12}$ & $\{m_0,m_1,m_3,m_4,m_6,m_8\}$ & $x_1^{10}x_2^{9}x_3^{5}z^{8}$ & $\{m_0,m_1,m_3,m_4,m_6,m_8\}$ \\
$x_1^{21}x_2^{15}x_3^{12}z^{16}$ & $\{m_0,m_1,m_3,m_4,m_5,m_7\}$ & $x_1^{24}x_2^{24}x_3^{12}z^{20}$ & $\{m_0,m_1,m_3,m_4,m_5,m_6\}$ \\
$x_1^{7}x_2^{6}x_3^{2}z^{5}$ & $\{m_0,m_1,m_2,m_5,m_7,m_8\}$ & $x_1^{10}x_2^{9}x_3^{2}z^{7}$ & $\{m_0,m_1,m_2,m_4,m_6,m_8\}$ \\
$x_1^{12}x_2^{9}x_3^{6}z^{9}$ & $\{m_0,m_1,m_2,m_4,m_5,m_8\}$ & $x_1^{18}x_2^{17}x_3^{10}z^{15}$ & $\{m_0,m_1,m_2,m_4,m_5,m_7\}$ \\
$x_1^{19}x_2^{15}x_3^{8}z^{14}$ & $\{m_0,m_1,m_2,m_3,m_6,m_7\}$ & $x_1^{12}x_2^{12}x_3^{12}z^{12}$ & $\{m_0,m_1,m_2,m_3,m_6,m_7\}$ \\
$x_1^{24}x_2^{24}x_3^{15}z^{21}$ & $\{m_0,m_1,m_2,m_3,m_5,m_7\}$ & $x_1^{13}x_2^{11}x_3^{6}z^{10}$ & $\{m_0,m_1,m_2,m_3,m_5,m_7\}$ \\
$x_1^{23}x_2^{20}x_3^{14}z^{19}$ & $\{m_0,m_1,m_2,m_3,m_5,m_6\}$ & $x_1^{20}x_2^{19}x_3^{15}z^{18}$ & $\{m_0,m_1,m_2,m_3,m_4,m_8\}$ \\
$x_1^{33}x_2^{27}x_3^{18}z^{26}$ & $\{m_0,m_1,m_2,m_3,m_4,m_7\}$ & $x_1^{13}x_2^{10}x_3^{4}z^{9}$ & $\{m_0,m_1,m_2,m_3,m_4,m_6\}$ \\
$x_1^{24}x_2^{21}x_3^{18}z^{21}$ & $\{m_0,m_1,m_2,m_3,m_4,m_5\}$ & $x_1^{15}x_2^{15}x_3^{6}z^{12}$ & $\{m_0,m_1,m_2,m_3,m_4,m_5\}$ \\
\end{longtable}
and
\begin{eqnarray*}
m_0= x_1^{6}x_2^{6}z^{4} & m_1= x_1^{6}x_2^{6}x_3^{6}z^{6} & m_2= x_1^{4}x_2^{4}x_3^{4}z^{4} \\
m_3= x_1^{6}x_2^{6}x_3^{3}z^{5} & m_4= x_1^{6}x_2^{3}x_3^{3}z^{4} & m_5= x_1^{4}x_2^{4}x_3 z^{3} \\
m_6= x_1^{5}x_2^{2}x_3^{2}z^{3} &m_7= x_1^{3} z &m_8= x_1^{2}x_2 z~.
\end{eqnarray*}

\bibliographystyle{alpha}
\bibliography{Bib}

\end{document}